\documentclass[11pt]{amsart}
\usepackage{amsmath,amssymb,amsthm}
\usepackage[latin1]{inputenc}
\usepackage{version,tabularx,multicol}
\usepackage{graphicx,float}
\usepackage{pstricks}
\usepackage{fancyhdr}

\newcommand{\reff}[1]{(\ref{#1})}

\theoremstyle{plain}
\newtheorem{theo}{Theorem}[section]
\newtheorem{cor}[theo]{Corollary}
\newtheorem{prop}[theo]{Proposition}
\newtheorem{lem}[theo]{Lemma}
\newtheorem{defi}[theo]{Definition}
\theoremstyle{remark}
\newtheorem{rem}[theo]{Remark}

\newcommand{\cl}{{\mathcal L}}

\newcommand{\cu}{{\mathcal U}}
\newcommand{\cv}{{\mathcal V}}

\newcommand{\E}{{\mathbb E}}

\renewcommand{\P}{{\mathbb P}}

\newcommand{\R}{{\mathbb R}}

\newcommand{\ind}{{\bf 1}}

\newcommand{\lima}{\lim_{c\rightarrow+\infty}}
\newcommand{\limab}{\underset{c\rightarrow+\infty}{\lim}}

\newcommand{\ud}{\mathrm{d}}
\newcommand{\Pp}{\mathbb{P}}

\newcommand{\tV}{\widetilde{\mathcal{V}}}
\newcommand{\ax}{c^{-\alpha}x}

\newcommand{\Vg}[1][c]{\underset{\leftarrow}{V}\hspace{-3pt}_{ m_{#1}}%
}
\newcommand{\Vd}[1][c]{\underset{\rightarrow}{V}\hspace{-3pt}_{ m_{#1}}%
}

\newcommand{\Ed}[1][c]{\underset{\rightarrow}{\E}^{#1}%
}

\newcommand{\uu}[1]{\underline{\underline{#1}}\hspace{0.01cm}}

\begin{document}

\title{Local time of a diffusion in a stable Lévy environment}
\author{Roland Diel} 

\address{
Roland Diel,
MAPMO, CNRS UMR 6628,
F\'ed\'eration Denis Poisson FR 2964,
Université d'Orléans,
B.P. 6759,
45067 Orléans cedex 2
FRANCE.
}
  
\email{roland.diel@univ-orleans.fr} 

\author{Guillaume Voisin}

\address{
Guillaume Voisin,
MAPMO CNRS UMR 6628, 
F\'ed\'eration Denis Poisson FR 2964,
Université d'Orléans,
B.P. 6759,
45067 Orléans cedex 2
FRANCE.}

\email{guillaume.voisin@univ-orleans.fr}

\begin{abstract}
Consider a one-dimensional diffusion in a stable Lévy environment. In this article we prove that the normalized local time process recentered at the bottom of the standard valley with height $\log t$, $(L_X(t,\mathfrak m_{\log t}+x)/t,x\in \R)$, converges in law to a functional of two independent Lévy processes, which are conditioned to stay positive. In the proof of the main result, we derive that the law of the standard valley is close to a two-sided Lévy process conditioned to stay positive. Moreover, we compute the limit law of the supremum of the normalized local time. In the case of a Brownian environment, similar result to the ones proved here have been obtained by Andreoletti and Diel \cite{ad:llltbd}.
\end{abstract}

\maketitle

\section{Introduction.}
Let $(V(x),x\in\R)$ be a real-valued càd-làg stochastic process such that $V(0)=0$. A \emph{diffusion in the environment} $V$ is a process $\left(X(V,t), t\geq0\right)$ (or $\left(X(t), t\geq0\right)$ if there is no ambiguity on the environment) whose generator given $V$
 is 
\begin{displaymath}
	\frac{1}{2}e^{V(x)}\frac{d}{dx}\left( e^{-V(x)}\frac{d}{dx} \right).
\end{displaymath}
Remark that if $V$ is almost surely differentiable and $V'$ is bounded, then $(X(t),t\geq0)$ is the solution of the following SDE:
\begin{displaymath}
\left\{ \begin{array}{ll}
\ud X(t)=\ud \beta(t)-\frac{1}{2}V'(X(t))\ud t,\\
X(0)=0%\qquad ,x\in\mathbb{R}
\end{array} \right.
\end{displaymath}
where $\beta$ is a standard Brownian motion independent of $V$. Historically, this model has been introduced by Schumacher  \cite{schumacher}. It was mainly studied when the environment $V$ is a Brownian motion (Brox \cite{b:oddpwm} and Kawazu and Tanaka \cite{kt:omdpdbe} ), a drifted Brownian motion (Kawazu and Tanaka \cite{kt:adpbewd}, Hu, Shi and Yor \cite{hsy:rcddbp}, Devulder \cite{d:assdpdbp} and Talet \cite{t:atebmdbp}  ) or a process which has the same asymptotic behavior as a Brownian motion(Hu and Shi \cite{hs:ltsrwre,hs:lssrwre}), this last case allows to study the discrete analogue of the diffusion in a Brownian environment, the so-called Sinai's random walk. More recently, general Lévy environments have been studied by Carmona \cite{c:mvbmrlp}, Tanaka \cite{t:lpcspdre}, Cheliotis \cite{c:oddare} and Singh \cite{s:rctdsnlp}.

The local time process $(L_X(t,x),t\geq 0,x\in\R)$ is defined by the occupation time formula, i.e., it is the unique process, almost surely continuous in the variable $t$, such that for all bounded Borel function $f$ and for all $t\geq0$,
\begin{equation}\label{toccu}
	\int_0^t f(X_s)\ud s=\int_\R f(x)L_X(t,x)\ud x.
\end{equation}
The existence of such a process can be easily proved, see e.g. Shi \cite{s:ltcre}. In the Brownian case, the local time process was largely studied see e.g. Hu and Shi \cite{hs:ltsrwre,hs:pmvsre}, Shi \cite{s:ltcre}, Devulder \cite{d:mltdpdbp} and Andreoletti and Diel \cite{ad:llltbd}. In this paper, we generalize the result of \cite{ad:llltbd} to the Lévy environment case. However several difficulties arise: contrary to the Brownian motion, the law of a Lévy process is not stable by time inversion; moreover the processes $(V(x)-\inf_{0\leq y \leq x}V(y),x\geq0)$, $(\sup_{0\leq y\leq x}V(y)-V(x),x\geq0)$ and $|V|$ do not have the same law.

In what follows, we focus on the local time of the diffusion when the environment $V$ is an $\alpha$-stable Lévy process on $\R$. We say that $V$ is $\alpha$-stable if there exists $\alpha>0$, such that for all $c>0$
$$ (c^{-1} V(c^\alpha x),x\in \R ) \stackrel{\cl}{=} (V(x),x\in \R) $$ 
where $ \stackrel{\cl}{=} $ is the equality in law. And $V$ is an $\alpha$-stable Lévy process if on top of that it is a Lévy process on $\R$.

One of the main results of this paper is the convergence in law of the supremum of the local time
\begin{displaymath}
\forall t\geq0,\quad L_X^*(t)=\sup_{x\in\R}L_X(t,x)
\end{displaymath}
which measures the time passed by the diffusion in the most visited point.

We next recall the notion of \emph{valley} introduced in the Brownian case by Brox \cite{b:oddpwm} (see also \cite{NevPit}). Denote by $\cv$ the space of càd-làg functions $\omega:\R \rightarrow \R$, endowed with the Skorohod topology (the space $\cv$ is then Polish). Let $c>0$, \label{defvalley} $\omega\in\mathcal{V}$ has a $c$-\textit{minimum} in the point $x_0$ if there exists $\xi$ and $\zeta$ such that $\xi<x_0<\zeta$ and 
\begin{itemize}
 
\item $\omega(\xi)\geq \left(\omega(x_0)\wedge \omega(x_0-)\right)+c$,
\item $\omega(\zeta-)\geq \left(\omega(x_0)\wedge \omega(x_0-)\right)+c\quad$ and
\item for all $x\in(\xi,\zeta)$, $x\neq x_0$, $\omega(x)\wedge \omega(x-)> \omega(x_0)\wedge \omega(x_0-)$.
\end{itemize}
Analogously, $\omega$ has a $c$-\textit{maximum} in $x_0$ if $-\omega$ has a $c$-minimum in $x_0$. Denote by $M_c(\omega)$ the set of $c$-\textit{extrema} of $\omega$. If $\omega$ takes pairwise distinct values in its local extrema and is continuous at these points, then we can check that $M_c(\omega)$ has no accumulation points and that the $c$-maxima and the $c$-minima alternate \label{valstand}. These assumptions are verified for $\alpha$-stable Lévy processes with $\alpha\in [1,2]$.  Thus there exists a triplet $\Delta_c=\Delta_c(\omega)=(\mathfrak p_c,\mathfrak m_c,\mathfrak q_c)$ of successive elements of $M_c$ such that 
\begin{itemize}\label{minvallee}
 \item $\mathfrak m_c$ and $0$ belongs to $[\mathfrak p_c,\mathfrak q_c]$,
\item $\mathfrak p_c$ and $\mathfrak q_c$ are $c$-maxima,
\item $\mathfrak m_c$ is a $c$-minimum.
\end{itemize}
This triplet is called the \textit{standard valley} with height $c$ of $\omega$. The stability of the environment $V$ allows us to reduce the study of a valley with height $c$ to a valley with height $1$. \label{stdval}

In the following, we denote
\begin{equation}\label{w+w-}
\left(\omega^+(x),x\geq0\right)=\left(\omega(x),x\geq0\right) \hbox{ , }\left(\omega^-(x),x\geq0\right) = \left(\omega((-x)-),x\geq0\right).
\end{equation}

Our main result gives the limit law of the local time process recentered on the minimum of the standard valley with height $\log t$.

\begin{theo}\label{cvloi}
Let $V$ be an $\alpha$-stable Lévy process with $\alpha\in [1,2]$ which is not a pure drift.\\
As $t\rightarrow \infty$, the process $\displaystyle \left(\frac{L_X(t,\mathfrak m_{\log t}+x)}{t},x\in\R\right)$
converges weakly in the Skorohod topology to the process $\displaystyle\left(\frac{e^{-\widetilde{V}(x)}}{\int_{-\infty}^{\infty}e^{-\widetilde{V}(y)}\ud y},x\in\R\right)$. The law of $\tilde{V}$ is $\tilde{\P}$ and under $\tilde{\P}(d\omega)$, $\omega^+$ and $\omega^-$ are independent, and distributed respectively as $\P^\uparrow$ and $\hat{\P}^\uparrow$, where $\P^\uparrow$ is the law of the Lévy process $V^+$ conditioned to stay positive, and $\hat{\P}^\uparrow$ is the law of the dual process $V^-$ conditioned to stay positive.
\end{theo}
Moreover, we obtain the localization of the favorite point $m^*(t)$ of the diffusion $X$.

\begin{theo}\label{cvptfav}
With the same assumptions as in previous theorem, for any $\delta \geq 0$,
 \begin{equation*}
\lim_{t\rightarrow+\infty} \mathcal{P}\left(\left|m^*(t)-\mathfrak{m}_{\log t}\right|\leq \delta \right)=1
 \end{equation*}
where $m^*(t)=\inf \big\{ x\in \R,\ L_X(t,x)=\sup \{ L_X(t,y), y\in\R\} \big\}$.
\end{theo}
The previous theorems imply the following corollary.

\begin{cor}\label{thlimsup}
With the same assumptions as before, we get
\begin{displaymath}
	\frac{L^*_X(t)}{t}   \xrightarrow[t\rightarrow \infty]{\mathcal L}     \frac{1 }{\int_{-\infty}^{\infty}e^{-\widetilde{V}(y)}\ud y} \label{eqthlimlaw}
\end{displaymath}
where $\widetilde{V}$ is the same process as in the previous theorem and $\xrightarrow{\mathcal L}$ denotes the convergence in law.
\end{cor}

\begin{proof} Recall that a Lévy process is continuous in probability. Thus, for any $x>0$, 
$$\mathcal{P}\big(V(x)=V(x-)\textrm{ and }V(-x)=V((-x)-)\big)=1.$$ 
Fix $\delta>0$. Theorem \ref{cvptfav} shows that it is sufficient to study the supremum of the process $x\mapsto L_X(t,\mathfrak{m}_{\log t}+x )$ on the compact set $[-\delta, \delta]$. %because $m^*(t)\in[\mathfrak{m}_{\log t}-\delta,\mathfrak{m}_{\log t}+\delta]$ $\mathcal P$-a.s. when $t\rightarrow +\infty$
Proposition VI.2.4 of \cite{js:ltsp} tells us that the function $F\mapsto \sup_{[-\delta,\delta]}F$ defined on $\mathcal V$ is continuous for the Skorohod topology for $F$ such that $F(\pm \delta)=F((\pm \delta)-)$. From Theorem \ref{cvloi} the result follows.
\end{proof}

The proof of Theorems \ref{cvloi} and \ref{cvptfav} is based on the study of the law of the environment. Indeed, for large times the diffusion is located in a valley whose right slope looks like the Lévy process $V^+$ conditioned to stay positive and the left slope like $V^-$ conditioned to stay positive (see Proposition \ref{vallee} for a rigorous statement). And, for large $t$ its local time process, normalized and centered at $\mathfrak m_{\log t}$, behaves like $e^{- V_{\mathfrak m_{\log t}}} / \int e^{- V_{\mathfrak m_{\log t}}}$ where $V_{\mathfrak m_{\log t}}$ is the environment recentered at $\mathfrak m_{\log t}$ (see Theorem \ref{limL} for more details).

The law of the valley has already been obtained by Tanaka \cite{t:ltoddpbe} in the special case of a Brownian environment; in this case, the right and left slopes are both 3-dimensional Bessel processes. In \cite{t:lpcspdre}, Tanaka proved similar results in the Lévy environment case.

This paper is organized as follows. In Section \ref{sec1}, we define the environment $V$ and we give the asymptotic law of the standard valley $V$ on $\R_+$ (Proposition \ref{convergence}) and in Section \ref{sec2}, we prove the convergence of the local time: Theorem \ref{cvloi} is a consequence of Theorem \ref{limL} and Proposition \ref{conv exp} and we also show the localization of the favorite point: Theorem \ref{cvptfav} is the last statement of Theorem \ref{limL}.

\section{standard valley of a Lévy process.}\label{sec1}

\subsection{Lévy process on $\R$.}\label{general}
 A $\cv$-valued random process $V$ is a Lévy process on $\R$ if its increments are independent and stationary (see \cite{c:mvbmrlp}) i.e.
\begin{enumerate}
\item $V(0)=0$,
\item if $x_0 < x_1 < \dots < x_n$ are real numbers, then the random variables $(V(x_{i+1})-V(x_i), 0\leq i \leq n-1)$ are independent,
\item for all $x$ and $y$, the law of $V(x+y)-V(x)$ only depends on $y$.
\end{enumerate}
Denote by $\P$ the law of the Lévy process $V$.

\begin{rem}
An equivalent definition of a Lévy process indexed by $\R$ is given by Cheliotis in \cite{c:oddare}: let $(V^+(t),t\geq0)$ be an $\R$-valued Lévy process starting from 0 and $(V^-(t), t\geq 0)$ a process independent of $V^+$ and with the same law as $-V^+$, we define $(V(t),t\geq 0 )$ by 
\begin{equation}\label{def levy}
\forall t\in\R,\ V(t)=\ind_{t\geq 0}V^+(t) + \ind_{t\leq 0}V^-((-t)-) .
\end{equation}
This process has càd-làg paths and satisfies the three assumptions of the previous definition. Conversely, if the three assumptions hold for a process, then it can be decomposed in two Lévy processes as in Formula \reff{def levy}.
\end{rem}

Define $\underline{V}^+$ the infimum process of $V^+$ and $\overline{V}^+$ the supremum process: for any $t\geq0$, $\underline{V}^+(t)=\inf_{0\leq s \leq t}V^+(s)$ and $\overline{V}^+(t)=\sup_{0\leq s \leq t}V^+(s)$. Let $\mathcal{N}$ be the excursion measure of the process $V^+-\underline{V}^+$ away from 0. We denote by $L$ a local time of $V^+-\underline{V}^+$ in 0 and $L^{-1}(t)=\inf\{ s\geq0 ;  L(s)>t \}$ its right continuous inverse.

Denote by $\cv^+$ the set of càd-làg functions $\omega:\R_+ \rightarrow \R$ endowed with the Skorohod topology. For $c>0$ fixed, we define the first passage time above a level $c$ of $\omega$ and the last previous time in which $\omega$ reaches $0$: for all $\omega\in\cv^+$,
$$\tau_c(\omega) = \inf\{ t\geq 0 ; \omega(t) \geq c \} \hbox{ and } n_c(\omega) = \sup\{ t\leq \tau_c(\omega) ; \omega(t) = 0 \hbox{ or }\omega(t-) = 0\}.$$
We then denote  
$$\underline{\tau}^+_c = \tau_c\left( V^+-\underline{V}^+ \right) \hbox{ , } m^+_c = n_c\left( V^+-\underline{V}^+ \right) \hbox{ and }$$
\begin{align*}
 J^+_c&=\left(V^+(m^+_c)+c\right)\vee\overline{V}^+(m^+_c).
\end{align*}
Similarly, we define analogue variables $\underline{\tau}^-_c$, $m^-_c$ and $J^-_c$ using the process $V^-$. 
\begin{center}
\begin{figure}[ht]
\caption{Standard valley with height $c$}

\scalebox{1} % Change this value to rescale the drawing.
{
\begin{pspicture}(0,-2.24)(15.221875,2.24)
\psline[linewidth=0.02cm,arrowsize=0.05291667cm 2.0,arrowlength=1.4,arrowinset=0.4]{<-}(7.06,2.23)(7.08,-2.23)
\psline[linewidth=0.02cm,arrowsize=0.05291667cm 2.0,arrowlength=1.4,arrowinset=0.4]{->}(0.0,-0.55)(14.86,-0.57)
\psline[linewidth=0.02](7.08,-0.59)(7.28,-0.21)(7.38,-0.43)(7.56,-0.01)
\psline[linewidth=0.02](7.56,0.21)(7.68,0.53)(7.86,0.17)(7.94,0.37)(8.16,-0.03)(8.24,0.13)(8.44,-0.17)(8.52,-0.01)
\psline[linewidth=0.02](9.02,-0.61)(9.38,-1.01)(9.54,-0.79)(9.88,-1.17)(10.0,-1.01)(10.42,-1.47)(10.78,-1.05)(10.96,-1.25)
\psline[linewidth=0.02](10.94,-0.91)(11.06,-1.05)(11.52,-0.47)(11.72,-0.69)(11.98,-0.35)
\psline[linewidth=0.02](11.98,0.01)(12.2,-0.19)(12.58,0.27)(12.74,0.09)(13.02,0.43)
\psline[linewidth=0.02](13.02,1.15)(13.24,0.95)(13.54,1.37)(13.82,1.09)(13.96,1.25)
\psline[linewidth=0.01cm,linestyle=dashed,dash=0.16cm 0.16cm](9.4,-1.49)(14.6,-1.51)
\psline[linewidth=0.01cm,linestyle=dashed,dash=0.16cm 0.16cm](13.04,1.77)(13.04,-1.99)
\psline[linewidth=0.01cm,linestyle=dashed,dash=0.16cm 0.16cm](9.44,0.75)(14.64,0.73)
\psline[linewidth=0.01cm,linestyle=dashed,dash=0.16cm 0.16cm](10.42,1.83)(10.42,-1.93)
\usefont{T1}{ptm}{m}{n}
\rput(13.491406,-0.38){$\underline{\tau}^+_c$}
\usefont{T1}{ptm}{m}{n}
\rput(10.771406,-0.3){$m_c^+$}
\psline[linewidth=0.01cm,linestyle=dashed,dash=0.16cm 0.16cm,arrowsize=0.05291667cm 2.0,arrowlength=1.4,arrowinset=0.4]{<->}(14.04,0.75)(14.06,-1.53)
\usefont{T1}{ptm}{m}{n}
\rput(14.261406,0.02){$c$}
\psline[linewidth=0.02](8.54,-0.21)(8.78,-0.51)(8.92,-0.25)(9.02,-0.39)
\psline[linewidth=0.02](7.06,-0.57)(6.96,-0.77)(6.74,-0.41)(6.66,-0.53)
\psline[linewidth=0.02](6.64,-0.27)(6.5,-0.07)(6.38,-0.21)(6.18,0.07)
\psline[linewidth=0.02](6.18,0.25)(6.02,0.09)(5.8,0.37)(5.54,0.03)
\psline[linewidth=0.02](5.58,-0.37)(5.44,-0.57)(5.36,-0.43)(5.1,-0.81)
\psline[linewidth=0.02](5.06,-1.21)(4.92,-1.07)(4.62,-1.47)(4.52,-1.35)(4.3,-1.67)(3.94,-1.23)(3.84,-1.35)(3.6,-1.05)
\psline[linewidth=0.02](3.6,-0.79)(3.44,-0.57)(3.3,-0.77)(2.96,-0.31)
\psline[linewidth=0.02](2.18,0.79)(1.96,1.07)(1.72,0.71)(1.64,0.81)(1.38,0.49)
\psline[linewidth=0.01cm,linestyle=dashed,dash=0.16cm 0.16cm](4.3,1.63)(4.3,-2.13)
\psline[linewidth=0.01cm,linestyle=dashed,dash=0.16cm 0.16cm](2.18,1.67)(2.18,-2.09)
\psline[linewidth=0.01cm,linestyle=dashed,dash=0.16cm 0.16cm](1.02,-1.67)(5.88,-1.67)
\psline[linewidth=0.01cm,linestyle=dashed,dash=0.16cm 0.16cm](1.02,0.61)(5.78,0.61)
\psline[linewidth=0.01cm,linestyle=dashed,dash=0.16cm 0.16cm,arrowsize=0.05291667cm 2.0,arrowlength=1.4,arrowinset=0.4]{<->}(1.56,0.61)(1.56,-1.69)
\usefont{T1}{ptm}{m}{n}
\rput(4.5,-0.3){$-m_c^-$}
\usefont{T1}{ptm}{m}{n}
\rput(2,-0.3){$-\underline{\tau}^-_c$}
\usefont{T1}{ptm}{m}{n}
\rput(1.2814063,0.1){$c$}
\psline[linewidth=0.02](2.96,-0.51)(2.68,-0.19)(2.54,-0.39)(2.18,0.05)
\end{pspicture} 
}

\end{figure}
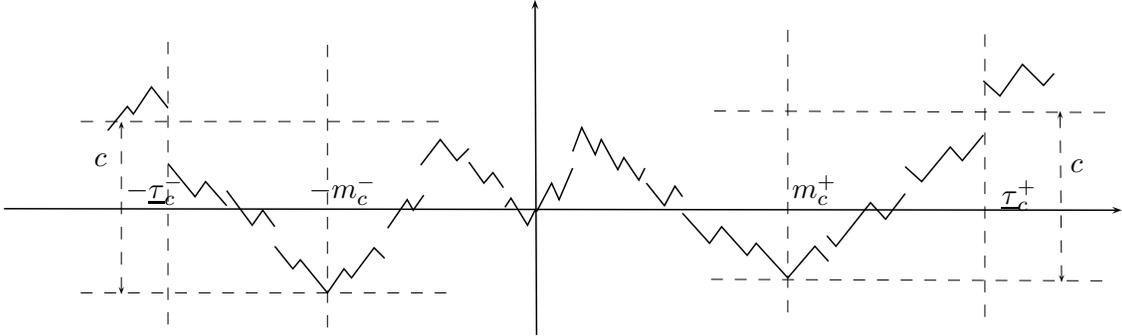
\end{center}
\begin{rem}\label{m=m-oum+}
We show that the minimum of the standard valley with height $c$ defined in the introduction satisfies $\mathfrak m_c=m^+_c$ if $J^+_c<J^-_c$ and $\mathfrak m_c=-m^-_c$ otherwise (see \cite{kesten} for a proof in the Brownian case which is also true in our case). Thus to study the process $V$ in a neighborhood of $\mathfrak m_c$, it is sufficient to study $V^+$ in a neighborhood of $m^+_c$ and $V^-$ in a neighborhood of $m^-_c$.
\end{rem}

Therefore, in this section we study the Lévy process $V^+$ defined on $\R_+$. However, we still denote $\P$ its law and until the end of this section, all the probability measures are measures on $\cv^+$. In order to simplify notations, we write $V$, $m_c$ and $\underline{\tau}_c$ for $V^+$, $m_c^+$ and $\underline{\tau}^+_c$.

\subsection{Valley of the Lévy process on $\R_+$.}\label{loivallee}

We say that $(V(t),0\leq t\leq \underline{\tau}_c)$ is the standard valley with height $c$ of the Lévy process on $\R_+$.

Furthermore, we say that $0$ is regular for $(0,+\infty)$ (resp. for $(-\infty,0)$) if $\P( \inf\{ t>0 ; V(t)>0\}=0)=1$ (resp. $\P( \inf\{ t>0 ; V(t)<0\}=0)=1$). We will only work with Lévy process for which $0$ is regular for $(-\infty,0)$ and $(0,+\infty)$. Thereby we recall some properties of such processes stated in \cite{c:oddare}, which give us the existence of the triplet $(\mathfrak p_c,\mathfrak m_c,\mathfrak q_c)$ defined in the introduction.

For $\omega\in\cv^+$, the point $x_0$ is called a left local maximum (resp. minimum) if there exists $\varepsilon>0$ such that $\omega(x)\leq \omega(x_0-)$ (resp. $\omega(x)\geq \omega(x_0-)$) for all $x\in(x_0-\varepsilon,x_0)$. Similarly, we define a right local maximum and minimum.

\begin{lem}\label{continuite}(\cite{c:oddare}, Lemma 3)\\
Let $V$ be a Lévy process such that 0 is regular for $(-\infty,0)$ and $(0,+\infty)$, then $\P$-a.s.:
\begin{enumerate}
\item $V$ is continuous at its left local extrema and at its right local extrema,
\item In no two local extrema $V$ has the same value.
\end{enumerate}
\end{lem}

The point $(2)$ of the previous Lemma gives us the uniqueness of the minimum $m_c$ such that $V(m_c)=\inf_{[0,\underline{\tau}_c]}V$. Moreover, $m_c$ is a local minimum, thus according to point $(1)$ the process $V$ is continuous in $m_c$, that means $ V(m_c)= V(m_c-)$.

\begin{rem}\label{continuite ext locaux}
If $\omega\in\cv^+$ is càd-làg on a bounded interval $I$, then there exists $x,y\in\bar{I}$ such that $\sup_{I}\omega = \omega(x)\vee \omega(x-)$ and $\inf_{I}\omega = \omega(y)\wedge \omega(y-)$. So if $\omega$ is continuous at its local extrema, we get:
$$\exists x\in \bar{I},\hbox{ s.t. }\sup_{I}\omega = \omega(x) \hbox{ and } \exists y\in \bar{I},\hbox{ s.t. }\inf_{I}\omega = \omega(y).$$
\end{rem}

\subsubsection{The pre-infimum and post-infimum processes.}

We now define the pre-infimum and post-infimum processes $(\Vg(t), 0\leq t \leq m_c)$ and $(\Vd(t), 0\leq t \leq \underline{\tau}_c-m_c)$ by
$$\Vg(t) = V((m_c -t)-) - V(m_c)\ \textrm{ and }\ \Vd(t) = V(m_c +t) - V(m_c).$$
When there is no possible mistake, we write $\underset{\leftarrow}{V}$ instead of $\Vg$ and $\underset{\rightarrow}{V}$ instead of $\Vd$. We also denote respectively $\underset{\leftarrow}{\P}^c$ and $\underset{\rightarrow}{\P}^c$ the laws of the processes $\underset{\leftarrow}{V}$ stopped at time $m_c$ and $\underset{\rightarrow}{V}$ stopped at time $\tau_c(\underset{\rightarrow}{V})$.

We now study the law of these two processes. Therefore, we need the notion of Lévy process conditioned to stay positive.

\subsubsection{Lévy process conditioned to stay positive.}\label{deflevypositif}

We present here Bertoin's construction by concatenation of the excursions of the Lévy process above $0$ (see Section 3, \cite{b:stiehlrwlp}).

The process $V$ is a semi-martingale, its local continuous martingale part is null or proportional to a standard Brownian motion. Denote by $\ell$ the local time in a semi-martingale sense of $V$ at 0, see for example \cite{p:sisde} and consider the time passed by $V$ in $(0,\infty)$ and its right continuous inverse,
$$A_t^+ = \int_0^t ds \ind_{V(s)>0} \hbox{ and }\alpha_t^+ = \inf\{ s\geq0 ; A_s^+>t \}.$$
We define the process $V^{\uparrow}$ by
$$
 V^{\uparrow}(t) =\left\{\begin{array}{lr}
  V(\alpha_t^+) + \frac{1}{2} \ell_{\alpha_t^+} + \sum_{0<s\leq \alpha_t^+} (0\vee V(s-)) \ind_{V(s) \leq 0} - (0\wedge V(s-))\ind_{V(s) >0}& \hbox{if}\ t<A_\infty^+,\\
 +\infty& \textrm{otherwise.}
 \end{array}\right.
$$
We denote by $\P^{\uparrow}$ the law of the process $V^{\uparrow}$ and $\hat{\P}^{\uparrow}$ the law of the dual process $\hat{V}=-V$ conditioned to stay positive.
\begin{rem}\label{def:vuparrow}
We will mainly use this expression to prove that when the process $V$  is stable, $V^\uparrow$ is stable too.
\end{rem}
We state some properties of the Lévy process conditioned to stay positive (see \cite{b:stiehlrwlp}):
\begin{itemize}
\item $\P$-a.s. for all $t>0$, $V^\uparrow(t)>0$,
\item in the case where $V$ has no Brownian part, $\ell=0$,
\item the process $V^\uparrow$ tends to $+\infty$ as $t$ goes to $+\infty$.
\end{itemize}

 For every  process $Y$ which is lower bounded, we define the future infimum: for all $t\geq0$, $\uu{Y}(t)=\inf_{s\geq t}Y(s)$. For $c>0$, we introduce the time
$$\hat{m}^\uparrow_c = n_c(\hat{V}^{\uparrow} - \uu{\hat{V}}^{\uparrow}) .$$

We write $\P^{\uparrow,c}$ (resp. $\hat{\P}^{\uparrow,c}$) for the law of the process $V^{\uparrow}$ (resp. $\hat{V}^{\uparrow}$) stopped at time $\tau_c(V^\uparrow)$ (resp. $\hat{m}^\uparrow_c$).

\subsubsection{Valley with height $c$.}
Unlike the Brownian case where the law of the post-infimum process is equal to the law of a 3 dimensional Bessel process, in the Lévy case, this law is only absolutely continuous with respect to the law of a Lévy process conditioned to stay positive. In order to have the law of the standard valley, we use Proposition 4.7 of Duquesne \cite{d:pdrlp} and therefore we have to introduce some notations. 

For $t\geq0$, define $U(t)=-V(L^{-1}(t))$ if $L(\infty)>t$ and $U(t)=+\infty$ otherwise (recall that $L$ is the local time of $V-\underline{V}$ in $0$). The process $(L^{-1},U)$ is called the ladder process, it is a subordinator. We then define the potential measure $\cu$ associated with $U$: for all measurable positive function $F$,
$$\int_{\R}\cu(dx)F(x) = \E \left[ \int_0^{L(\infty)} dv F(U(v)) \right].$$

For the excursion measure $\mathcal{N}$ and for all measurable non-negative function $f$ on $\cv$, we write
$$\mathcal{N}(f)=\int f(v)\mathcal{N}(dv).$$

\begin{prop}\label{vallee}
We consider a Lévy process $V$, which does not go to $+\infty$ neither to $-\infty$ and such that 0 is regular for $(0,+\infty)$ and for $(-\infty,0)$.

Then the pre-infimum and post-infimum processes $(\underset{\leftarrow}{V}(t), 0\leq t \leq m_c)$ and $( \underset{\rightarrow}{V}(t) , 0\leq t \leq \tau_c(\underset{\rightarrow}{V})) $ are independent. Moreover

\begin{enumerate}
\item The law  $\underset{\leftarrow}{\P}^{c}$ of the pre-infimum process $(\underset{\leftarrow}{V}(t), 0\leq t \leq m_c)$ is equal to the law $\hat{\P}^{\uparrow,c}$ of the process $(\hat{V}^{\uparrow} (t), 0\leq t \leq \hat{m}^\uparrow_c)$.

\item The law $\underset{\rightarrow}{\P}^c$ of the post-infimum process $(\underset{\rightarrow}{V}(t),0\leq t \leq \tau_c(\underset{\rightarrow}{V}) )$ is absolutely continuous with respect to the law $\P^{\uparrow,c}$ of the process $(V^{\uparrow} (t), 0\leq t \leq \tau_c(V^\uparrow))$.

More precisely, $\displaystyle \underset{\rightarrow}{\P}^{c}(d\omega) = f_c\left(\omega(\tau_c(\omega))\right) \P^{\uparrow,c}(d\omega)$

where for all $x\in\R_+$, $\frac{1}{f_c(x)}=\cu \left( [0, x) \right) \mathcal{N}(\tau_c(v) < \infty)$.
\end{enumerate}
\end{prop}

\begin{rem}\hspace{0.1cm}\label{rem 1ere prop}
\begin{enumerate}
\item When $V$ is a Lévy process with no positive jumps, $\omega(\tau_c(\omega))=c$, $\P^{\uparrow,c}-a.s$. Thus, $\omega\mapsto f_c(\omega(\tau_c(\omega)))$ is constant and equal to $1$ and the post-infimum process is equal in law to the process $V^\uparrow$ conditioned to stay positive.

\item $\P$-a.s., $\underset{\leftarrow}{V}$ is continuous in its local extrema (see Lemma \ref{continuite}), then if $c$ tends to $+\infty$, using point (1) of  Proposition \ref{vallee}, $\hat{V}^\uparrow$ and $V^\uparrow$ are continuous at their local extrema too.

\end{enumerate}
\end{rem}

\begin{proof}[Proof of the Proposition \ref{vallee}]

$(1)$  We begin with the proof of the first part of Proposition \ref{vallee} and the independence of the two processes for $c=1$. To simplify the notations, we will not write the index $1$. 

We want to express the pre-infimum process $\underset{\leftarrow}{V}$ in terms of the excursions of $V$ above its minimum and of the  associated jump of $\underline{V}$. We consider the couple of processes $(V-\underline{V} , -\underline{V})$ and we give an expression using excursions. We use the proof of Lemma 4 of \cite{b:sdtplspi}, remark that the result is still true for every Lévy process that does not converge to $-\infty$, see  Theorem 27 of \cite{d:ftlp}.

In order to use excursion theory, we denote by $\bar{\cv}$ the set of the càd-làg paths $v:[0,\infty)\rightarrow \R_+$ such that
\begin{itemize}
\item[(i)] $v(t) >0$ for $0<t<\sigma(v) = \inf\{ s>0, v(s)=0 \}$,
\item[(ii)] $v(0)=v(t)=0$ for $t\geq\sigma(v)$.
\end{itemize}

We consider the excursion process of $V-\underline{V}$ above 0 defined by:
\begin{align*}
p_1(t) &= 
\left\{
\begin{array}{ll}
r\mapsto(V-\underline{V})(L^{-1}(t-)+r)\ind_{[0,L^{-1}(t)-L^{-1}(t-)]}(r) & \hbox{ if } L^{-1}(t-)<L^{-1}(t)\\
0 & \hbox{ if } L^{-1}(t-)=L^{-1}(t)
\end{array}
\right. \\
\hbox{and the} & \hbox{ jumps process of $\underline{V}$ defined by:}\\
p_2(t) &=  \underline{V}(L^{-1}(t-)) - \underline{V}(L^{-1}(t)).
\end{align*}

Thus the process $p(\cdot)=(p_1(\cdot),p_2(\cdot))$ is a $\bar{\cv}\times \R_+$-valued Poisson point process (see for example \cite{gp:filpsm}). Remark that $p_1\left( L(\underline{\tau})\right)$ is the excursion starting from $m$, that is to say that $L(\underline{\tau})$ is the first time $t$ where the excursion $p_1(t)$ has a height bigger than $1$. Thus, the processes $(p(t) ; 0\leq t < L(\underline{\tau}) )$ and $(p(t) ; t\geq L(\underline{\tau}) )$ are independent. The pre-infimum process depends only on the first and the post-infimum process depends only on the second, then these two processes are independent too.

In order to obtain the process $\underset{\leftarrow}{V}$, we first invert the excursions of $V-\underline{V}$ then we invert the time in each of these excursions.

Denote by $\tilde{p}_i = \left(p_i(L(\underline{\tau})-t)\ind_{t<  L(\underline{\tau})} + p_i(t)\ind_{t \geq L(\underline{\tau})} \, ; t\geq0\right)$ the process $p_i$ where the order of the excursions before $L(\underline{\tau})$ is changed for $i=1$ and $i=2$. The law of the Poisson point process $p$ is stable by inversion of the order of its excursions before $L(\underline{\tau})$, we deduce the equality in law of the two processes $(\tilde{p}_1(t),\tilde{p}_2(t))$ and $(p_1(t),p_2(t))$.

Consider now the time inversion in the excursions defined for $v\in \bar{\cv}$ by 
$$[v](s) = 
\begin{cases}
v((\sigma(v)-s)-) &\hbox{ for }  0\leq s \leq \sigma(v)\\
0 & \hbox{ for }  s> \sigma(v).
\end{cases}
$$

Using this transformation in the previous law, the processes $([\tilde{p}_1(t)],\tilde{p}_2(t)\, ; t\geq0)$ and $([p_1(t)],p_2(t)\, ; t\geq0)$ have same law. With the first one we get the process $\underset{\leftarrow}{V}$. Indeed, $ \left([\tilde{p}_1(t)]\, ; t\geq0 \right)$ is the excursion process of $\underset{\leftarrow}{V}$ above its future infimum and $ \left(\tilde{p}_2(t)\, ; t\geq0 \right)$ represents the jumps of the future infimum of $\underset{\leftarrow}{V}$. We now use the couple $([p_1(t)],p_2(t)\, ; t\geq0)$ to get a Lévy process conditioned to stay positive.

\label{procr}We recall the notations of \cite{b:sdtplspi}, for $t\geq 0$, we use $g(t) = \sup\{ s<t ; V(s)=\underline{V}(s) \}$ and $d(t) = \inf\{ s>t ; V(s)=\underline{V}(s) \}$, the left and right extremities of the excursion interval of $V-\underline{V}$ away from 0 that contains $t$ and
$$\mathcal R_{V-\underline{V}}(t) =
\begin{cases}
(V-\underline{V})((d(t)+g(t)-t)-) & \hbox{ if  } g(t)<d(t) \\
0 & \hbox{otherwise}.
\end{cases}
$$

See Figure \ref{fig:V} and \ref{fig:RV}.

The process $L$ is also a local time in 0 for the process $\mathcal R_{V-\underline{V}}$. So the excursion process of $\mathcal R_{V-\underline{V}}$  above 0, parametrized by $L$, is the Poisson point process $([p_1(t)],t\geq0)$ and, similarly, the excursion process of $\left( (V-\underline{V})((m-t)-) \ind_{0\leq t < m} + \mathcal R_{V-\underline{V}}(t) \ind_{t\geq m} , t\geq0\right)$ above 0 is  the Poisson point process $([\tilde{p}_1(t)],t\geq0)$. We have an analogue equality for the jumps associated with every excursion. Then we get the equality in law of the processes

\begin{flushleft}
$\displaystyle \left( (V-\underline{V})((m-t)-) \ind_{0\leq t < m} + \mathcal R_{V-\underline{V}}(t) \ind_{t\geq m} ,\right. $
\end{flushleft}

\begin{flushright}
$\displaystyle  \left.(\underline{V}((m -t)-) - \underline{V}( m ))\ind_{0\leq t < m} - \underline{V}(d(t)) \ind_{t\geq m} \; ; \; t\geq 0\right)$
\end{flushright}

\begin{center}
and
\end{center}

$$\left( \mathcal R_{V-\underline{V}}(t)  , -\underline{V}(d(t) ; t \geq 0) \right) .$$

Adding the two parts for $t<m$ we get: 
\begin{align*}
\left( \underset{\leftarrow}{V}(t) ; 0\leq t < m\right) &\stackrel{\cl}{=} \left( \mathcal R_{V-\underline{V}}(t)  -\underline{V}(d(t)) ; 0\leq t < m\right),
\end{align*}

See Figure \ref{fig:RV} and \ref{fig:preinf}.

%------------------premier dessin----------------------%

\begin{center}
\begin{figure}[ht]
\caption{The Lévy process $(V(t);t\geq0)$}\label{fig:V}

\scalebox{1} % Change this value to rescale the drawing.
{
\begin{pspicture}(0,-2.235)(9.745,2.235)
\psline[linewidth=0.01cm,arrowsize=0.05291667cm 2.0,arrowlength=1.4,arrowinset=0.4]{<-}(0.0,2.23)(0.02,-2.23)
\psline[linewidth=0.01cm,arrowsize=0.05291667cm 2.0,arrowlength=1.4,arrowinset=0.4]{->}(0.02,-0.63)(9.74,-0.63)
\psline[linewidth=0.024cm](3.24,-1.13)(4.08,-1.13)
\psline[linewidth=0.02](3.26,-1.11)(3.4,-0.57)(3.52,-0.95)
\psline[linewidth=0.02](3.52,-0.67)(3.66,-0.25)(3.74,-0.49)(3.78,-0.33)(3.96,-0.81)(4.06,-0.59)
\psline[linewidth=0.024cm](4.06,-1.27)(5.68,-1.27)
\psline[linewidth=0.02](4.06,-1.27)(4.2,-0.97)(4.28,-1.11)(4.46,-0.73)
\psline[linewidth=0.02](4.46,-0.97)(4.56,-1.11)(4.76,-0.63)(4.86,-0.79)(5.02,-0.39)
\psline[linewidth=0.02](5.02,-0.73)(5.1,-0.53)(5.28,-0.81)(5.34,-0.65)(5.58,-0.93)(5.66,-0.77)
\psline[linewidth=0.02](6.2,-1.47)(6.4,-0.97)(6.52,-1.19)(6.74,-0.67)
\psline[linewidth=0.02](6.74,-0.33)(6.98,-0.69)(7.18,-0.17)(7.3,-0.37)(7.5,0.17)
\psline[linewidth=0.02](7.5,0.73)(7.74,0.25)(7.82,0.43)(8.12,-0.07)
\psline[linewidth=0.02](8.12,-0.33)(8.4,-0.75)(8.5,-0.61)(8.86,-1.13)
\psline[linewidth=0.024cm](5.66,-1.47)(8.84,-1.47)
\psline[linewidth=0.016cm,linestyle=dashed,dash=0.16cm 0.16cm](5.82,0.43)(8.84,0.41)
\psline[linewidth=0.016cm,linestyle=dashed,dash=0.16cm 0.16cm](6.2,0.85)(6.2,-2.01)
\psline[linewidth=0.016cm,linestyle=dashed,dash=0.16cm 0.16cm](7.5,0.85)(7.46,-2.11)
\psline[linewidth=0.016cm,linestyle=dashed,dash=0.16cm 0.16cm,arrowsize=0.05291667cm 2.0,arrowlength=1.4,arrowinset=0.4]{<->}(8.68,0.41)(8.68,-1.47)
\usefont{T1}{ptm}{m}{n}
\rput(8.82375,-0.545){\tiny $1$}
\usefont{T1}{ptm}{m}{n}
\rput(7.70375,-0.545){\tiny $\underline{\tau}$}
\usefont{T1}{ptm}{m}{n}
\rput(6.33375,-0.545){\tiny $m$}
\psline[linewidth=0.024cm](0.02,-0.61)(1.8,-0.61)
\psline[linewidth=0.02](0.02,-0.61)(0.22,0.13)(0.36,-0.31)(0.46,-0.03)
\psline[linewidth=0.02](0.46,0.19)(0.62,-0.11)(0.7,0.11)(0.94,-0.35)(1.0,-0.15)
\psline[linewidth=0.02](0.98,0.03)(1.06,0.27)(1.3,-0.23)(1.38,0.01)(1.56,-0.35)(1.62,-0.15)(1.8,-0.47)
\psline[linewidth=0.024cm](1.8,-0.81)(3.26,-0.81)
\psline[linewidth=0.02](1.8,-0.79)(2.04,-0.19)(2.12,-0.39)
\psline[linewidth=0.02](2.12,-0.55)(2.24,-0.27)(2.38,-0.55)(2.46,-0.37)(2.54,-0.49)
\psline[linewidth=0.02](2.54,-0.23)(2.64,-0.05)(2.92,-0.47)(3.0,-0.31)(3.18,-0.57)(3.26,-0.43)
\psline[linewidth=0.016cm,linestyle=dashed,dash=0.16cm 0.16cm](1.82,0.23)(1.8,-1.29)
\psline[linewidth=0.016cm,linestyle=dashed,dash=0.16cm 0.16cm](3.26,0.23)(3.24,-1.31)
\psline[linewidth=0.016cm,linestyle=dashed,dash=0.16cm 0.16cm](2.46,0.83)(2.44,-1.71)
\usefont{T1}{ptm}{m}{n}
\rput(3.45375,0.015){\tiny $d(t)$}
\usefont{T1}{ptm}{m}{n}
\rput(1.95375,0.035){\tiny $g(t)$}
\usefont{T1}{ptm}{m}{n}
\rput(2.56375,0.115){\tiny $t$}
\psline[linewidth=0.024](5.66,-1.47)(5.8,-1.19)(5.86,-1.29)(5.98,-1.07)
\psline[linewidth=0.024](5.96,-1.39)(6.04,-1.25)(6.18,-1.45)
\end{pspicture} 
}

\end{figure}
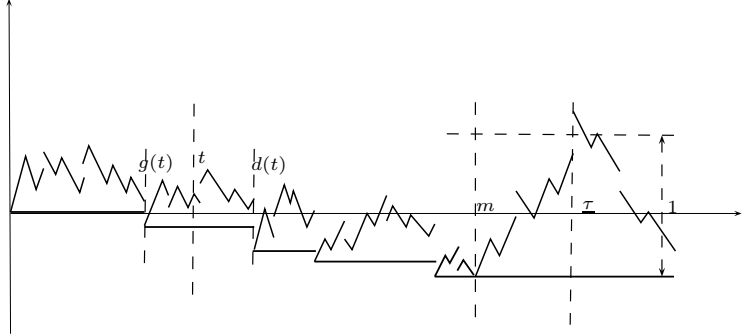
\end{center}

%----------------fin dessin 1-------------------

%----------------debut dessin 2----------------------

\begin{center}
\begin{figure}[ht]
\caption{The process $\left( \mathcal R_{V-\underline{V}}(t)  -\underline{V}(d(t)); t\geq 0\right)$}\label{fig:RV}

\scalebox{1} % Change this value to rescale the drawing.
{
\begin{pspicture}(0,-2.375)(9.725,2.38)
\psline[linewidth=0.01cm,arrowsize=0.05291667cm 2.0,arrowlength=1.4,arrowinset=0.4]{<-}(0.02,2.09)(0.04,-2.37)
\psline[linewidth=0.01cm,arrowsize=0.05291667cm 2.0,arrowlength=1.4,arrowinset=0.4]{->}(0.0,-0.87)(9.72,-0.87)
\psline[linewidth=0.024cm](4.08,-0.188)(3.24,-0.188)
\psline[linewidth=0.02](4.050482,-0.19998996)(3.9121473,0.36009675)(3.793575,-0.034038343)
\psline[linewidth=0.02](3.793575,0.25637698)(3.6552408,0.692)(3.5761926,0.44307256)(3.5366685,0.60902417)(3.3588102,0.11116932)(3.26,0.33935282)
\psline[linewidth=0.024cm](5.7,-0.048)(4.08,-0.048)
\psline[linewidth=0.02](5.7,-0.048)(5.56,0.252)(5.48,0.112)(5.3,0.492)
\psline[linewidth=0.02](5.3,0.252)(5.2,0.112)(5.0,0.592)(4.9,0.432)(4.74,0.832)
\psline[linewidth=0.02](4.74,0.492)(4.66,0.692)(4.48,0.412)(4.42,0.572)(4.18,0.292)(4.1,0.452)
\psline[linewidth=0.02](8.861875,0.17)(8.661875,0.67)(8.541875,0.45)(8.321875,0.97)
\psline[linewidth=0.02](8.321875,1.31)(8.081875,0.95)(7.881875,1.47)(7.761875,1.27)(7.561875,1.81)
\psline[linewidth=0.02](7.561875,2.37)(7.321875,1.89)(7.241875,2.07)(6.941875,1.57)
\psline[linewidth=0.02](6.941875,1.31)(6.661875,0.89)(6.561875,1.03)(6.24,0.55)
\psline[linewidth=0.024cm](8.981875,0.17)(5.68,0.19)
\psline[linewidth=0.016cm,linestyle=dashed,dash=0.16cm 0.16cm](9.241875,2.07)(6.221875,2.05)
\psline[linewidth=0.024cm](1.82,-0.668)(0.04,-0.668)
\psline[linewidth=0.02](1.82,-0.668)(1.62,0.072)(1.48,-0.368)(1.38,-0.088)
\psline[linewidth=0.02](1.38,0.132)(1.22,-0.168)(1.14,0.052)(0.9,-0.408)(0.84,-0.208)
\psline[linewidth=0.02](0.86,-0.027999999)(0.78,0.212)(0.54,-0.288)(0.46,-0.048)(0.28,-0.408)(0.22,-0.208)(0.04,-0.528)
\psline[linewidth=0.024cm](3.28,-0.528)(1.82,-0.528)
\psline[linewidth=0.02](3.28,-0.508)(3.04,0.092)(2.96,-0.108)
\psline[linewidth=0.02](2.96,-0.268)(2.84,0.012)(2.7,-0.268)(2.62,-0.088)(2.54,-0.208)
\psline[linewidth=0.02](2.54,0.052)(2.44,0.232)(2.16,-0.188)(2.08,-0.027999999)(1.9,-0.288)(1.82,-0.148)
\psline[linewidth=0.02cm,linestyle=dashed,dash=0.16cm 0.16cm,arrowsize=0.05291667cm 2.0,arrowlength=1.4,arrowinset=0.4]{<->}(8.52,2.07)(8.52,0.17)
\usefont{T1}{ptm}{m}{n}
\rput(8.671406,1.22){\tiny $1$}
\psline[linewidth=0.02cm,linestyle=dashed,dash=0.16cm 0.16cm](6.22,1.55)(6.22,-0.85)
\usefont{T1}{ptm}{m}{n}
\rput(6.041406,-1.12){\tiny $m$}
\psline[linewidth=0.024](6.22,0.19)(6.08,0.47)(6.02,0.37)(5.9,0.59)
\psline[linewidth=0.024](5.92,0.27)(5.84,0.41)(5.7,0.21)
\end{pspicture} 
}

\end{figure}
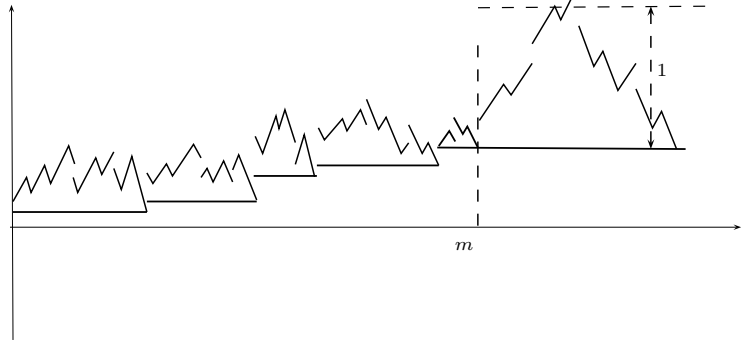
\end{center}

%--------------fin dessin 2 -------------------

%-----------------dessin 3 ---------------------

\begin{center}
\begin{figure}[ht]
\caption{The pre-infimum process    } \label{fig:preinf}
%  $(V((m-t)-)-V(m) ; t\geq0)$   ou    $(\underset{\leftarrow}{V}(t) ; t\geq0)$
\scalebox{1} % Change this value to rescale the drawing.
{
\begin{pspicture}(0,-2.245)(9.755,2.248)
\psline[linewidth=0.01cm,arrowsize=0.05291667cm 2.0,arrowlength=1.4,arrowinset=0.4]{<-}(0.010000031,2.22)(0.030000031,-2.24)
\psline[linewidth=0.01cm,arrowsize=0.05291667cm 2.0,arrowlength=1.4,arrowinset=0.4]{->}(0.030000031,-0.64)(9.75,-0.64)
\psline[linewidth=0.016cm,linestyle=dashed,dash=0.16cm 0.16cm](6.17,2.24)(6.19,-0.96)
\usefont{T1}{ptm}{m}{n}
\rput(6.48375,-0.515){\tiny $m$}
\psline[linewidth=0.024cm](2.9520001,-0.32)(2.112,-0.32)
\psline[linewidth=0.02](2.932,-0.3)(2.792,0.24)(2.672,-0.14)
\psline[linewidth=0.02](2.672,0.14)(2.532,0.56)(2.4520001,0.32)(2.412,0.48)(2.232,0.0)(2.132,0.22)
\psline[linewidth=0.024cm](2.132,-0.46)(0.512,-0.46)
\psline[linewidth=0.02](2.132,-0.46)(1.992,-0.16)(1.9120001,-0.3)(1.732,0.08)
\psline[linewidth=0.02](1.732,-0.16)(1.6320001,-0.3)(1.432,0.18)(1.332,0.02)(1.172,0.42)
\psline[linewidth=0.02](1.172,0.08)(1.092,0.28)(0.91200006,0.0)(0.85200006,0.16)(0.61200005,-0.12)(0.532,0.04)
\psline[linewidth=0.024cm](6.19,0.2)(4.392,0.2)
\psline[linewidth=0.02](6.172,0.2)(5.972,0.94)(5.8320003,0.5)(5.732,0.78)
\psline[linewidth=0.02](5.732,1.0)(5.572,0.7)(5.492,0.92)(5.252,0.46)(5.192,0.66)
\psline[linewidth=0.02](5.212,0.84)(5.132,1.08)(4.892,0.58)(4.812,0.82)(4.632,0.46)(4.572,0.66)(4.392,0.34)
\psline[linewidth=0.024cm](4.392,0.0)(2.932,0.0)
\psline[linewidth=0.02](4.392,0.02)(4.152,0.62)(4.072,0.42)
\psline[linewidth=0.02](4.072,0.26)(3.9520001,0.54)(3.812,0.26)(3.732,0.44)(3.652,0.32)
\psline[linewidth=0.02](3.652,0.58)(3.552,0.76)(3.272,0.34)(3.1920002,0.5)(3.012,0.24)(2.932,0.38)
\psline[linewidth=0.024](0.532,-0.66)(0.39200002,-0.38)(0.33200002,-0.48)(0.21200003,-0.26)
\psline[linewidth=0.024](0.23200004,-0.58)(0.15200002,-0.44)(0.012000032,-0.64)
\psline[linewidth=0.024cm](0.51000005,-0.64)(0.012000032,-0.64)
\end{pspicture} 
}

\end{figure}
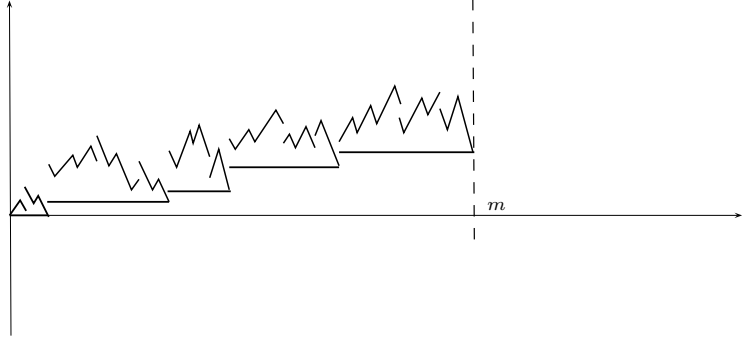
\end{center}

%-----------------------fin dessin 3--------------------------

We now use Theorem 28 of \cite{d:ftlp} : under $\P$, the law of the process $\left( \mathcal R_{V-\underline{V}}(t)  -\underline{V}(d(t)) ,t\geq0\right)$ is $\hat{\P}^{\uparrow}$.

Now we express the final time $m$ in terms of the process $\left( \mathcal R_{V-\underline{V}}(t)  -\underline{V}(d(t)), t\geq 0\right)$.

Recall that $\underline{\tau} = \inf\{ t\geq 0 ; V(t)-\underline{V}(t) \geq 1 \}$ and that $m = \sup\{ t\leq \underline{\tau} ; V(t) - \underline{V}(t) = 0 \}$. We define the variable $\underline{\tau}' = \tau(  \mathcal R_{V-\underline{V}}) = \inf\{t\geq 0 ; \mathcal R_{V-\underline{V}}\geq 1    \}$. Notice that $\underline{\tau}$ and $\underline{\tau}'$ belong to the same excursion interval of $V-\underline{V}$. The variable $m$ remains stable by time inversion in every excursions of $V-\underline{V}$ and by the inversion of the excursions before $m$, then $m=\sup\{t\leq \underline{\tau}' ; \mathcal R_{V-\underline{V}}(t)=0\}$. We write 
$$Y(t) := \mathcal R_{V-\underline{V}}(t) - \underline{V}(d(t)).$$
We prove that
$$\mathcal R_{V-\underline{V}}(t) = Y(t) - \uu{Y}(t),$$
it is enough to prove that, for a fixed $t\geq0$, the post-infimum process $\uu{Y}(t) = \inf_{s\geq t}Y(s)$ is equal to $- \underline{V}(d(t))$.

On the interval $[t,d(t))$, the process $(\underline{V}(d(s)),t\leq s < d(t))$ is constant and $$\inf_{t\leq s < d(t)} \mathcal R_{V-\underline{V}}(s) = \mathcal R_{V-\underline{V}}(d(t)-)=0$$ then $\inf_{t\leq s < d(t)}Y(s) = - \underline{V}(d(t))$. The process $\mathcal R_{V-\underline{V}}$ is non-negative and vanishes at times  $\{d(t)-,t\geq 0\}$ (see Corollary 1, \cite{r:nirlp}), and the process $- \underline{V}(d(.))$ increases by jumps at times $\{d(t),t\geq 0\}$, then $\uu{Y}(t)=\inf_{t\leq s < d(t)}Y(s)$ and so $\uu{Y}(t) =- \underline{V}(d(t))$. We deduce that 

$$m=\sup\{t\leq \underline{\tau}' ; Y(t) - \uu{Y}(t)=0\}.$$

Finally, under $\P$, the law of the process $\left( \mathcal R_{V-\underline{V}}(t)  -\underline{V}(d(t)) \right)_{0\leq t\leq m}$ is the law of the process  $\hat{V}^{\uparrow}$ stopped at time $\hat{m}^{\uparrow}$ and the point (1) is proved.

\vspace{0.5cm}

$(2)$ Now we consider the post-infimum process. Let $F$ be a bounded measurable function on $\cv^+$. We can write 
$$\underset{\rightarrow}{\E}^{c} \left[ F\left( \omega  \right)\right] = \mathcal{N} \left( F\left( v(\cdot \wedge \tau_c)  \right) | \tau_c(v) < \infty \right).$$

We use Proposition 4.7 of \cite{d:pdrlp}: for all bounded measurable function $G$ on $\cv^+$, 

\begin{align*}
\E^{\uparrow,c}\left[ G(\omega) \right] &= \mathcal{N}\left( G\left( v(\cdot \wedge \tau_c) \right)  \cu \left( [0, v(\tau_c) ) \right) | \tau_c <\infty \right) \mathcal{N}\left( \tau_c(v) <\infty \right)\\
&= \mathcal{N}\left( \frac{G( v(\cdot\wedge \tau_c) )}{ f_c(v(\tau_c))} \Big | \tau_c(v) <\infty \right)
\end{align*}
where we recall that for all $x\in\R_+$, $\frac{1}{f_c(x)} = \cu \left( [0, x) \right) \mathcal{N}(\tau_c(v) < \infty)$. Using the function $G( v(\cdot\wedge \tau_c) ) = F\left( v(\cdot \wedge \tau_c)  \right) f_c(v(\tau_c))$, we obtain the second point of the proposition.

\end{proof}

\subsection{Stable Lévy process}

\begin{defi}(\cite{b:pl}, Section VIII)\\
The Lévy process $V$ is stable with index $\alpha \in (0,2]$ if, for all $c>0$
$$ \left(c^{-1} V(c^\alpha t),t\geq 0\right) \stackrel{\cl}{=} \left(V(t),t\geq 0\right) .$$
\end{defi}

For $\alpha\in(0,1)\cup(1,2)$, the Fourier transform of an $\alpha$-stable Lévy process $V$ is : for all $t\geq0$ and for all $\lambda\in \R$, $\E \left[ e^{i\lambda V(t)}\right] = e^{-t\psi(\lambda)}$ with
$$\psi(\lambda) = k |\lambda|^\alpha \left(1-i \beta \hspace{0,2cm} sgn(\lambda)\tan(\frac{\pi \alpha}{2}) \right)$$
where $k>0$ and $\beta\in[-1,1]$.\\
The Lévy measure of the process $V$ is absolutely continuous with respect to the Lebesgue measure, it can be written as follows:
$$\pi(dx) = c^+ x^{-\alpha-1} \ind_{x > 0} dx + c^-  x ^{-\alpha-1} \ind_{x<0} dx$$
where $c^+$ and $c^-$ are positive numbers non equal to 0 together and such that $\beta = \frac{c^+ - c^-}{c^+ + c^-}$.

In the quadratic case ($\alpha=2$), $V$ is a Brownian motion and its characteristic exponent can be written $\psi(\lambda) = k \lambda^2$ where $k$ is a positive number. The Lévy measure is then trivial.

In the case $\alpha=1$, $V$ is a Cauchy process and $\psi(\lambda)=k|\lambda| + id\lambda$ with  $k>0$ and $d\in \R$. The Lévy measure is then proportional to $x^{-2} dx$.

\begin{rem}\label{stabmc} If the process $V$ is $\alpha$-stable, then the pre-infimum processes $\left(c^{-1}\Vg(c^\alpha t),t\geq0\right)$ and $\Vg[1]$ (stopped at correct times) have the same law. This also holds for the post-infimum processes $\left(c^{-1}\Vd(c^\alpha t),t\geq0\right)$ and $\Vd[1]$. 
\end{rem}

\begin{lem}Let $V$ be an $\alpha$-stable Lévy process which is not a pure drift. Then $V$ is recurrent if and only if $\alpha \geq1$.
\end{lem}

\begin{proof}We use a  condition of recurrence given in \cite{b:pl}, Theorem I.17: $V$ is recurrent if and only if $\int_{-r}^r \Re \left[ \frac{1}{\psi(\lambda)}\right]d\lambda = +\infty$ where $\Re[z]$ is the real part of the complex number $z$.

For $\alpha=1$, $\int_{-r}^r \frac{k|\lambda|}{(k^2+d^2)\lambda^2}d\lambda = +\infty$ because $k\neq0$.

For $\alpha\neq 1$, $\int_{-r}^r \frac{1}{k|\lambda|^\alpha(1+\beta^2 \tan^2(\frac{\pi \alpha}{2}))}d\lambda<+\infty$ if and only if $\alpha <1$. 
\end{proof}

Then we use a stable Lévy process with index $\alpha\in [1,2]$ and which is not a pure drift. We need others properties for $V$, mainly to use the results of Section \ref{loivallee}.

\begin{lem}
Let $V$ be a stable Lévy process with index $\alpha\in [1,2]$ and which is not a pure drift. Then,
\begin{itemize}
\item 0 is regular for $(0,+\infty)$ and for $(-\infty,0)$,
\item $\P$-a.s., $\underset{t\rightarrow +\infty}{\lim\sup}V(t) =+\infty$ and  $\underset{t\rightarrow +\infty}{\lim\inf}V(t) = -\infty$.
\end{itemize}
\end{lem}

\begin{proof}\hspace{1cm}

$\bullet$ Rogozin criterion (see Prop. VI.11 of \cite{b:pl})  gives a necessary and sufficient condition to have the regularity in 0:

0 is regular for  $(0,+\infty)$ (resp. for $(-\infty,0)$) if and only if $\displaystyle \int_0^1 t^{-1}\P(V(t)\geq 0)dt = +\infty$ (resp. $\displaystyle \int_0^1 t^{-1}\P\left(V(t) \leq 0\right)dt = +\infty)$.

Remark, using the stability property of $V$, that $\P(V(t)\geq 0)$ does not depend on $t$. Moreover this probability is positive if $V$ is not a pure drift, see \cite{b:pl}, VIII.1. 

With the previous criterion, $0$ is regular for $(0,+\infty)$ and for $(-\infty,0)$.

$\bullet$ We have $\int_1^\infty t^{-1}\P(V(t)\geq0)dt = \int_1^\infty t^{-1}\P(V(t)\leq0)dt = +\infty$, this implies that $\P$-a.s., $\underset{t\rightarrow +\infty}{\lim\sup}V(t) =+\infty$ and  $\underset{t\rightarrow +\infty}{\lim\inf}V(t) = -\infty$, see \cite{b:pl}, Theo. VI.12.

\end{proof}

\begin{lem}\label{stabiliteuparrow}
We consider a stable Lévy process $V$ with index $\alpha\in(0,2]$. Then the process $V^{\uparrow}$ is an $\alpha$-stable process.
\end{lem}

\begin{proof}
In the case where $\alpha=2$, the stable Lévy process $V$ is a standard Brownian motion, $V^{\uparrow}$ is a $3$-dimensional Bessel process (see \cite{mck:ensd}) and then is stable with index $2$.

If $\alpha\in(0,2)$, the stable Lévy process has no Brownian part. Recall that $\alpha_t^+$ is the right continuous inverse of $A_t^+=\int_0^t \ind_{V(s)>0} ds $. These two processes are stable with index $1$. We use the definition given in Section \ref{deflevypositif}:
$$V^{\uparrow}(t) = V(\alpha_t^+) + \sum_{0<s\leq \alpha_t^+} \left(0\vee V(s-) \right)\ind_{V(s) \leq 0} - \left(0\wedge V(s-)\right)\ind_{V(s) >0}, $$
We deduce that the process $V^{\uparrow}$ is stable with index $\alpha$.

\end{proof}

\begin{prop}\label{convergence}
We consider a stable Lévy process with index $\alpha\in [1,2]$ and which is not a pure drift. When $c$ converges to $+\infty$:
\begin{enumerate}
\item the law $\underset{\leftarrow}{\P}^{c}$ converges weakly (in the Skorohod sense) to the law $\hat{\P}^{\uparrow}$.

\item the law $\underset{\rightarrow}{\P}^{c}$ converges weakly (in the Skorohod sense) to the law $\P^{\uparrow}$.

\end{enumerate}
\end{prop}

To prove the second point, we need the following lemma.

\begin{lem}\label{levy renormalise}
We fix $\varepsilon>0$ and we define 
$$
\sigma_\varepsilon=\sigma_\varepsilon(V^\uparrow)=\inf\left\{t>0/V^\uparrow(t)-\uu{V}^\uparrow(t)=0 \textrm{ and } \exists s<t, V^\uparrow(s)-\uu{V}^\uparrow(s)>\varepsilon\right\}.
$$
Then the process $\left(V^{\uparrow}_{\sigma_\varepsilon}(t),t\geq0\right)=\left(V^\uparrow(\sigma_\varepsilon+t)-V^\uparrow(\sigma_\varepsilon),t\geq0\right)$ is independent of the process $\left(V^\uparrow(t),0\leq t\leq \sigma_\varepsilon\right)$ and its law is given by $\P^\uparrow$.
\end{lem}
\begin{proof}

Thanks to Theorem 28 of \cite{d:ftlp}, the law of the process $\mathcal{R}_{\overline{V}-V}+V(d(\cdot))$ is $\P^\uparrow$ (where $\mathcal{R}$ and $d$ are defined like page \pageref{procr} but using the process $\overline{V}-V$). It is enough to prove the result in the case where $V^\uparrow=\mathcal{R}_{\overline{V}-V}+V(d(\cdot))$. But in this case, $\sigma_\varepsilon$ can be written as
$$
\sigma_\varepsilon=\inf\left\{t>0/\overline{V}(t)-V(t)=0 \textrm{ and } \exists s<t, \overline{V}(s)-V(s)>\varepsilon\right\}
$$
then $\sigma_\varepsilon$ is a stopping time for $V$. 

Thus the process $\left(V_{\sigma_\varepsilon}(t),t\geq0\right)=\left(V(\sigma_\varepsilon+t)-V(\sigma_\varepsilon),t\geq0\right)$ is a Lévy process with law $\P$ independent of $\left(V(t),0\leq t\leq \sigma_\varepsilon\right)$. We can check that  $V^{\uparrow}_{\sigma_\varepsilon}(t)=\mathcal{R}_{\overline{V}_{\sigma_\varepsilon}-V_{\sigma_\varepsilon}}(t)+V_{\sigma_\varepsilon}(d(t))$, this proves the Lemma.
\end{proof}

\begin{proof}[Proof of Proposition \ref{convergence}]

$(1)$ Thanks to Proposition \ref{vallee}, the laws of $\underset{\leftarrow}{\P}^{c}$ and $\hat{\P}^{\uparrow,c}$ are equal. When $c\rightarrow \infty$, $\hat{\P}^{\uparrow}$-a.s., $\hat{m}^\uparrow_c$ tends to infinity, then $\underset{\leftarrow}{\P}^{c}=\hat{\P}^{\uparrow,c}$ converges to $\hat{\P}^{\uparrow}$.

\vspace{.4cm}

\noindent(2) Now we prove that for a continuous bounded function $F$ in the Skorohod topology such that $F(\omega)=F(\omega\ind_{[-K,K]})$, we have
\[
\lim_{c\rightarrow \infty} \Ed\left[ F(\omega) \right]=\E^\uparrow\left[ F(\omega) \right].
\]

As $\P$-a.s., $\lim_{c\rightarrow\infty}\tau_c(\Vd)=\infty$, so $\lim_{c\rightarrow\infty} \underset{\rightarrow}{\P}^c(K\leq \tau_c(\omega))=1$ and thus
$$\lim_{c\rightarrow \infty} \Ed\left[ F(\omega) \right]=\lim_{c\rightarrow\infty}\Ed\left[ F(\omega), K\leq \tau_c(\omega) \right].$$
We first use the stability of $V$. Thanks to Remark \ref{stabmc},
\begin{align*}
\Ed\left[ F(\omega), K\leq \tau_c(\omega) \right]&= \Ed[1]\left[ F\left(  c\,\omega(c^{-\alpha}\cdot) \right) ,K\leq c^{\alpha} \tau_1(\omega)  \right],\\
&= \E^\uparrow\left[ F\left( c\,\omega(c^{-\alpha}\cdot)\right)f_1\left(\omega(\tau_1)\right) ,K\leq c^{\alpha}\tau_1(\omega)  \right]
\end{align*}
where we use the result (2) of Proposition \ref{vallee}, for the second equality. 

We now prove the asymptotic independence of $F\left( c\,\omega(c^{-\alpha}\cdot)\right)$ and of $f_1\left(\omega(\tau_1)\right)$ under $\P^\uparrow$. Then we express these two terms with respect to the processes $(\omega(t),0\leq t\leq \sigma_\varepsilon)$ and $(\omega(t),\sigma_\varepsilon\leq t)$ where, for $\varepsilon>0$, we recall that:
$$
\sigma_\varepsilon=\sigma_\varepsilon(\omega)=\inf\left\{t>0/\omega(t)-\uu{\omega}(t)  =0 \textrm{ and } \exists s<t, \omega(s)-\uu{\omega}(s)>\varepsilon\right\}.
$$

\begin{center}
\begin{figure}[ht]
\caption{The process $V^\uparrow$}

\label{vfleche}

\scalebox{1} % Change this value to rescale the drawing.
{
\begin{pspicture}(0,-2.015)(14.599063,2.015)
\psline[linewidth=0.01cm,arrowsize=0.05291667cm 2.0,arrowlength=1.4,arrowinset=0.4]{<-}(0.7,2.01)(0.7,-2.01)
\psline[linewidth=0.01cm,arrowsize=0.05291667cm 2.0,arrowlength=1.4,arrowinset=0.4]{->}(0.0,-1.09)(11.16,-1.09)
\psline[linewidth=0.024cm](3.642,-0.77)(2.802,-0.77)
\psline[linewidth=0.02](3.622,-0.75)(3.482,-0.21)(3.362,-0.59)
\psline[linewidth=0.02](3.362,-0.31)(3.222,0.11)(3.142,-0.13)(3.102,0.03)(2.922,-0.45)(2.822,-0.23)
\psline[linewidth=0.024cm](2.822,-0.91)(1.202,-0.91)
\psline[linewidth=0.02](2.822,-0.91)(2.682,-0.61)(2.602,-0.75)(2.422,-0.37)
\psline[linewidth=0.02](2.422,-0.61)(2.322,-0.75)(2.122,-0.27)(1.98,-0.57)(1.9,-0.37)
\psline[linewidth=0.02](1.88,-0.61)(1.782,-0.39)(1.602,-0.67)(1.542,-0.51)(1.302,-0.79)(1.222,-0.63)
\psline[linewidth=0.024cm](6.86,-0.25)(5.08,-0.25)
\psline[linewidth=0.02](6.862,-0.25)(6.662,0.49)(6.522,0.05)(6.422,0.33)
\psline[linewidth=0.02](6.422,0.55)(6.262,0.25)(6.182,0.47)(5.942,0.01)(5.882,0.21)
\psline[linewidth=0.02](5.902,0.39)(5.822,0.63)(5.582,0.13)(5.502,0.37)(5.322,0.01)(5.262,0.21)(5.082,-0.11)
\psline[linewidth=0.024cm](5.082,-0.45)(3.622,-0.45)
\psline[linewidth=0.02](5.082,-0.43)(4.9,-0.03)(4.78,-0.23)
\psline[linewidth=0.02](4.762,0.15)(4.642,0.43)(4.502,0.15)(4.4,0.41)(4.34,0.25)
\psline[linewidth=0.02](4.3,0.55)(4.18,0.83)(3.96,0.29)(3.86,0.45)(3.702,0.09)(3.622,0.23)
\psline[linewidth=0.024](1.222,-1.11)(1.082,-0.83)(1.022,-0.93)(0.902,-0.71)
\psline[linewidth=0.024](0.922,-1.03)(0.842,-0.89)(0.702,-1.09)
\psline[linewidth=0.024cm](1.2,-1.09)(0.702,-1.09)
\psline[linewidth=0.02](6.88,0.15)(7.08,0.53)(7.2,0.37)(7.36,0.67)
\psline[linewidth=0.02](7.36,0.95)(7.52,1.29)(7.64,1.13)(7.74,1.37)
\psline[linewidth=0.02](7.74,1.51)(7.96,1.05)(8.08,1.25)(8.32,0.79)(8.42,0.93)
\psline[linewidth=0.02](8.44,0.59)(8.68,0.21)(8.78,0.35)(8.98,0.01)
\psline[linewidth=0.024cm](9.0,-0.01)(6.9,-0.01)
\psline[linewidth=0.016cm,linestyle=dashed,dash=0.16cm 0.16cm,arrowsize=0.05291667cm 2.0,arrowlength=1.4,arrowinset=0.4]{<->}(4.12,0.65)(4.12,-0.47)
\usefont{T1}{ptm}{m}{n}
\rput(4.41375,-0.125){\tiny $\varepsilon$}
\psline[linewidth=0.016cm,linestyle=dashed,dash=0.16cm 0.16cm](5.08,1.17)(5.08,-1.59)
\psline[linewidth=0.016cm,linestyle=dashed,dash=0.16cm 0.16cm](0.58,0.99)(9.66,0.97)
\psline[linewidth=0.016cm,linestyle=dashed,dash=0.16cm 0.16cm](7.34,1.69)(7.38,-1.69)
\usefont{T1}{ptm}{m}{n}
\rput(5.40375,-1.265){\tiny $\sigma_\varepsilon$}
\usefont{T1}{ptm}{m}{n}
\rput(7.70375,-1.365){\tiny $\tau^\uparrow_1$}
\psline[linewidth=0.016cm,linestyle=dashed,dash=0.16cm 0.16cm](5.08,-0.45)(9.78,-0.45)
\psline[linewidth=0.016cm,linestyle=dashed,dash=0.16cm 0.16cm,arrowsize=0.05291667cm 2.0,arrowlength=1.4,arrowinset=0.4]{<->}(9.1,0.97)(9.14,-0.45)
\psline[linewidth=0.016cm,linestyle=dashed,dash=0.16cm 0.16cm,arrowsize=0.05291667cm 2.0,arrowlength=1.4,arrowinset=0.4]{<->}(9.14,-0.47)(9.14,-1.09)
\usefont{T1}{ptm}{m}{n}
\rput(9.68375,-0.785){\tiny $V^\uparrow(\sigma_\varepsilon)$}
\usefont{T1}{ptm}{m}{n}
\rput(10.25375,0.255){\tiny $V^{\uparrow}_{\sigma_\varepsilon} ({\tau^\uparrow_{1-V^\uparrow(\sigma_\varepsilon)}})$}
\end{pspicture} 
}

\end{figure}
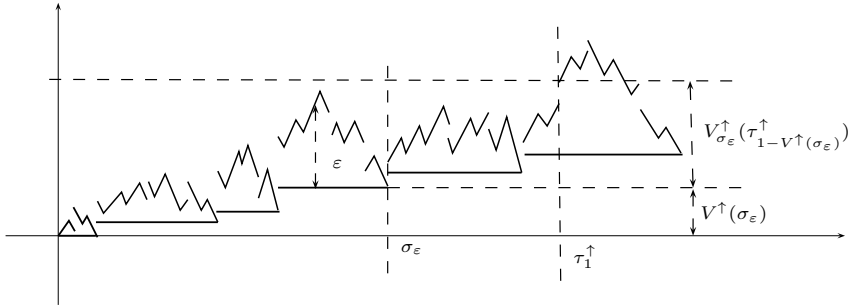
\end{center}

We fix $\varepsilon>0$, we remark that $\lim_{c\rightarrow +\infty} \P^\uparrow\left( K \leq c^{\alpha} \tau_1 \right)=\lim_{c\rightarrow +\infty} \P^\uparrow\left( K\leq c^{\alpha}\sigma_\varepsilon  \right)=1$. Then we can replace $K \leq c^{\alpha} \tau_1$ by $K\leq c^{\alpha}\sigma_\varepsilon$ in the expectation. 

Moreover, $\P^\uparrow$-a.s., $\lim_{\varepsilon\rightarrow 0}\sigma_\varepsilon=0$. Indeed, if there exists $\delta>0$ such that for all $\varepsilon>0$, $\sigma_\varepsilon>\delta$, then the excursion of $\omega-\uu{\omega}$ containing time $\delta$ is the first one. And this is $\P^\uparrow$-a.s. impossible because 0 is regular for the process $V$. Thus $\lim_{\varepsilon \rightarrow 0} \P^\uparrow\left( \sigma_\varepsilon < \tau_1\right)=1$. By conditioning with respect to $\left( \omega(t), 0\leq t \leq \sigma_\varepsilon\right)$ and using that $\{ \sigma_\varepsilon<\tau_1^\uparrow \}=\{\forall t\leq \sigma_\varepsilon, V_t^\uparrow <1\}$ is measurable with respect to  $\left( \omega(t), 0\leq t \leq \sigma_\varepsilon\right)$, we get: 

\begin{flushleft}
\begin{equation}\label{hypotheses}
 \E^\uparrow\left[ F\left(  c\,\omega(c^{-\alpha}\cdot)\right)f_1\left(\omega(\tau_1)\right) ,  c^{-\alpha}K\leq \sigma_\varepsilon < \tau_1 \right] = \hspace{5cm}
 \end{equation}
\end{flushleft}
 
\begin{flushright}
 \begin{equation*}
  \hspace{3cm}\E^\uparrow\left[ F\left(  c\,\omega(c^{-\alpha}\cdot)\right) \, \ind_{c^{-\alpha}K\leq\sigma_\varepsilon< \tau_1} \, \E^\uparrow\left[ f_1\left(\omega(\tau_1)\right) \Big | \left(\omega(t),0\leq t \leq \sigma_\varepsilon\right)\right] \right].
  \end{equation*}
\end{flushright}

We compute the limit of the conditional expectation when $\varepsilon$ converges to 0. When $\sigma_\varepsilon < \tau_1$, we can write $\omega(\tau_1) = \omega(\sigma_\varepsilon) + \omega_{\sigma_\varepsilon} ({\tau_{1-\omega(\sigma_\varepsilon)}})$ where $\omega_{\sigma_\varepsilon}$ is the path starting from $\sigma_\varepsilon$ : $\left(\omega_{\sigma_\varepsilon}(t),t\geq0\right)=\left(\omega(\sigma_\varepsilon+t)-\omega(\sigma_\varepsilon),t\geq0\right)$, see Figure \ref{vfleche}. Thanks to Lemma \ref{levy renormalise},

\begin{align*}
\E^\uparrow\left[ f_1\left(\omega(\tau_1)\right) \Big | \left( \omega(t),0\leq t \leq \sigma_\varepsilon \right) \right] = \chi(\omega(\sigma_\varepsilon))
\hbox{ where } \chi(x) = \E^\uparrow\left[f_1\left(  x + \omega(\tau_{1-x})  \right) \right].
\end{align*}
Remark that $\P^\uparrow$-a.s, $\lim_{\varepsilon\rightarrow 0}\omega(\sigma_\varepsilon)=0$. Prove that $\P^\uparrow$-a.s., $\lim_{y\rightarrow 1-}\omega(\tau_y)=\omega(\tau_1)$.  $\P^\uparrow$-a.s., $\omega$ is continuous at its local extrema (see (2) of Remark \ref{rem 1ere prop}). Thus, thanks to Remark \ref{continuite ext locaux}, $\P^\uparrow$-a.s., $\sup_{[0,\tau_1)}\omega$ is attained in a point $x_0\in[0,\tau_1]$ and  $\omega(x_0)\leq 1$.

If $x_0=\tau_1$, then $\omega(\tau_1)=1$ and $\lim_{y\rightarrow 1-}1-\omega(\tau_y) \leq \lim_{y\rightarrow 1-}1-y =0$. And if $x_0<\tau_1$, then $\omega(x_0)<1$ and then, for $1-y$ small enough, $\tau_1=\tau_y$, thereby $\lim_{y\rightarrow 1-}\omega(\tau_y)=\omega(\tau_1)$.

Moreover, in the case where $V$ is a stable process with index $\alpha\in[1,2)$, an explicit expression of $\cu$ is given in Example 7 of \cite{dk:oulp}:
$$\cu\left( [  0, x )\right) = \frac{x^{\alpha\rho}}{\Gamma(\alpha \rho+1)} \hbox{ where }\rho=\P(V(t)\geq0)\hbox{ does not depend on $t$}.$$
Thus using the expression of $f_c$ given in Proposition \ref{vallee}, we get: 
\begin{equation}\label{expressionf1}
f_1(x)=\frac{\Gamma(\alpha\rho+1)}{x^{\alpha\rho}\mathcal{N}(\tau_1(\omega) < \infty)},
\end{equation}
then $f_1$ is continuous on  $\R_+^*$ and
\begin{eqnarray}\label{conv ki}
 \lim_{\varepsilon\rightarrow 0}\E^\uparrow\left[ f_1\left(\omega(\tau_1)\right) \Big | \left( \omega(t),0\leq t \leq \sigma_\varepsilon\right) \right]=\nonumber\lim_{\varepsilon\rightarrow 0} \chi(\omega(\sigma_\varepsilon)) &=&
 \E^\uparrow\left[ f_1\left(\omega(\tau_{1}) \right) \right],  \\
 &=& 1 \textrm{, $\P^\uparrow$-a.s.},\label{limite psi}
\end{eqnarray}
because $f_1\left(\omega(\tau_{1}) \right)$ is the density of $\underset{\rightarrow}{\P}^1$ with respect to $\P^{\uparrow,1}$.

Using the stability of the process $V^\uparrow$ given in Lemma \ref{stabiliteuparrow}, we get that for all $c\geq0$,
$$ \E^\uparrow\left[ F\left(  c\,\omega(c^{-\alpha}\cdot)\right)\right]= \E^\uparrow\left[ F\left(\omega\right)\right]=\E\left[ F\left(V^\uparrow \right)\right].$$
So to prove the proposition, it is enough to show that 

$$ \lim_{\varepsilon\rightarrow 0} \lim_{c\rightarrow \infty}\left|   \E^\uparrow\left[ F\left(  c\,\omega(c^{-\alpha}\cdot) \right) \, \ind_{c^{-\alpha}K\leq\sigma_\varepsilon}  \chi\left(\omega(\sigma_\varepsilon) \right)\ind_{\sigma_\varepsilon < \tau_1}\right]    -    \E^\uparrow\left[ F\left(  c\,\omega(c^{-\alpha}\cdot)\right)\right]   \right|=0.$$

We introduce $  \E^\uparrow\left[ F\left(  c\,\omega(c^{-\alpha}\cdot) \right)\ind_{c^{-\alpha}K\leq\sigma_\varepsilon} \right]  $, and we get the following upper bound :

\begin{flushleft}
$\displaystyle \left|   \E^\uparrow\left[ F\left(  c\,\omega(c^{-\alpha}\cdot) \right) \, \ind_{c^{-\alpha}K\leq\sigma_\varepsilon}  \chi\left(\omega(\sigma_\varepsilon) \right)\ind_{\sigma_\varepsilon < \tau_1}\right]    -    \E^\uparrow\left[ F\left(  c\,\omega(c^{-\alpha}\cdot)\right)\right]   \right| \leq$
\end{flushleft}

\begin{flushright}
$\displaystyle           \E^\uparrow\left[\left|  F\left(  c\,\omega(c^{-\alpha}\cdot) \right) \, \ind_{c^{-\alpha}K\leq\sigma_\varepsilon}  \left(\chi\left(\omega(\sigma_\varepsilon) \right)\ind_{\sigma_\varepsilon < \tau_1}-1\right)\right|\right]      +      \E^\uparrow\left[\left| F\left(  c\,\omega(c^{-\alpha}\cdot) \right)\left(\ind_{c^{-\alpha}K\leq\sigma_\varepsilon}-1\right)\right|  \right]                  .   $
\end{flushright}
For $\varepsilon$ fixed, we have $\lim_{c\rightarrow +\infty} \P^\uparrow\left(c^{-\alpha} K\leq \sigma_\varepsilon  \right)=1$. Then for the second part, we get :

$$\displaystyle  \lim_{c\rightarrow \infty} \E^\uparrow\left[  \left| F\left(  c\,\omega(c^{-\alpha}\cdot) \right)\left( \ind_{c^{-\alpha}K\leq\sigma_\varepsilon}  -    1\right) \right|  \right]            \leq  \left\| F \right\|_\infty  \lim_{c\rightarrow \infty}   \E^\uparrow\left[ \left|  \ind_{c^{-\alpha}K\leq\sigma_\varepsilon}  -    1 \right| \right]=0,$$
and for the first part, we have : 

$$\displaystyle \underset{c\rightarrow \infty}{\overline{\lim} }  \E^\uparrow\left[\left| F\left(  c\,\omega(c^{-\alpha}\cdot) \right) \, \ind_{c^{-\alpha}K\leq\sigma_\varepsilon}  \left(1-\chi\left(\omega(\sigma_\varepsilon) \right)\ind_{\sigma_\varepsilon < \tau_1}\right)\right|\right]  \leq   \displaystyle   \left\| F \right\|_\infty   \E^\uparrow\left[\left|  1- \chi\left(\omega(\sigma_\varepsilon) \right)\ind_{\sigma_\varepsilon < \tau_1} \right|  \right] . $$

Let $\varepsilon$ tends to 0, we have $\P^\uparrow$-a.s, $\lim_{\varepsilon\rightarrow 0}\sigma_\varepsilon=0$ and we use \reff{conv ki} to get the convergence : 

$$ \lim_{\varepsilon\rightarrow 0} \lim_{c\rightarrow \infty}   \E^\uparrow\left[\left| F\left(  c\,\omega(c^{-\alpha}\cdot) \right) \, \ind_{c^{-\alpha}K\leq\sigma_\varepsilon}  \left(1-\chi\left(\omega(\sigma_\varepsilon) \right)\ind_{\sigma_\varepsilon < \tau_1}\right)\right|\right]   =0.$$

This completes the proof.

\end{proof}

\section{Diffusion in a random environment}\label{sec2}
\subsection{Asymptotic behavior of a diffusion in a stable environment}
We suppose that the diffusion $X$ in the random environment $V$ is defined on a probability space $\left(\Omega,\mathcal{F},\mathcal{P}\right)$ and we write, as previously, $\Pp$ the law of $V$. But we take a general case for the environment: \emph{ in all this section, we only suppose that $V$ is an $\alpha$-stable càd-làg process on $\R$ ($\alpha>0$) and not necessary a Lévy process}. We will prove that the asymptotic behavior of $L_X$, the local time of the diffusion, depends only on the environment.

We first give some properties of $X$. We can prove (see for example \cite{s:swvsc}) that there exists a Brownian motion $B$ independent of $V$ such that $\mathcal{P}$-a.s.,
\begin{equation}\label{eqx}
 \forall t>0,\quad X(V,t)=S_V^{-1}\left(B\left(T_V^{-1}(t)\right)\right)
\end{equation}
where
\begin{equation*}
	\forall x \in \mathbb{R},\quad S_V(x):=\int_0^xe^{V(y)}\ud y \label{eqS}
\end{equation*}
and
\begin{equation*}
\forall t\geq 0,\quad T_{V}(t):=\int_0^te^{-2V(S_V^{-1}(B(s)))}\ud s.\label{eqT}
\end{equation*}
In order to simplify the notations, we write $S$ and $T$ respectively for $S_V$ and $T_{V}$, when there is no confusion.
From Formula $(\ref{eqx})$ and the occupation time Formula \reff{toccu}, we deduce an expression of the local time process of $X$ using the local time process $L_B$ of the Brownian motion $B$. For all $x\in\R$ and all $t\geq0$,
\begin{eqnarray}
	L_X(t,x)=e^{-V(x)}L_B(T^{-1}(t),S(x)).\label{eqLx} \label{tsc}
\end{eqnarray}

To study the asymptotic behavior of the diffusion, we first work in a quenched environment.

As Brox in the Brownian environment case  (see \cite{b:oddpwm}), we introduce the process $X_c(V,t):=X(c V,t)$. For all $x\in\R$, we write $V^c(x):=c^{-1}V(c^\alpha x)$. There is a direct link between the law of $X_{c}$ and the law of $X$:
\begin{lem} \label{egx}
For every fixed $c$ , conditionally on $V$, the processes $\displaystyle\left( c^{-\alpha}X(V,c^{2\alpha}t); t\geq 0\right)$ and $\displaystyle\left( X_c(V^c,t); t\geq 0 \right)$ have the same law.
\end{lem}

\begin{proof}
We can easily check the result with Formula $(\ref{eqx})$.
\end{proof}

With the definition of the local time (Formula \reff{toccu}), we get an analogous relation for the local time processes:
\begin{cor}\label{eglx}
For every fixed $c$, conditionally on $V$,\\
the processes $\displaystyle\left( L_{X_c(V^c,\cdot)}(t,x) ;t\geq0, x\in \R\right)$ and $\displaystyle\left( c^{-\alpha} L_{X(V,\cdot)}(c^{2\alpha}t,c^\alpha x); t\geq0, x\in \R\right)$ have the same law.
\end{cor}

We will first give a convergence result for the process $L_{X_c}$ and then we will deduce a convergence result for the process $L_{X}$. The asymptotic behavior of $L_{X_c}$ with a quenched environment can be only described with respect to the environment if this one is a "good" environment in the sense of the following definition: 
\begin{defi}
We call a good environment every path $\omega\in\mathcal{V}$ that checks
\begin{enumerate}
\item$\displaystyle\liminf_{x\rightarrow\pm\infty}\omega(x)=-\infty \textrm{ and }\limsup_{x\rightarrow\pm\infty}\omega(x)=+\infty$,
\item $\omega$ is continuous at its  local extrema and the values of $\omega$ in these points are pairwise distinct.
\end{enumerate}
Denote by $\tV$ the set of good environments.
\end{defi}

Remark that if the paths of $V$ are in  $\tV$, then the minimum $\mathfrak m_c(V)$ of the standard valley with height $c$ (see page \pageref{valstand}) is well defined.

Then we can define, for every $c>0$ and every $r>0$,
\begin{align*}
	&a_{cr}=a_{cr}(V_{\mathfrak m_c}):=\sup\left\{x\leq0/V_{\mathfrak m_c}(x)>cr\right\}\textrm{ and} \\
& b_{cr}=b_{cr}(V_{\mathfrak m_c}):=\inf\left\{x\geq0/V_{\mathfrak m_c}(x)>cr\right\}
\end{align*}
where for $x\in\R$, $(V_x(y), y\in\R)$ is the recentered environment in $x$: for all $y\in\R$,
\begin{displaymath}
	V_x(y):=V(x+y)-V(x). \label{shdipo}
\end{displaymath}
Now we can give the result in a quenched environment. We denote by $\mathcal{P}_V$ the probability measure $\mathcal{P}(\cdot|V)$.

\begin{prop}\label{cle}$ $\\
We fix a real number $r\in(0,1)$ and we choose an $\R$-valued function $h$ such that $\underset{c\rightarrow\infty}{\lim}h(c)=1$.
If $V$ is a good environment, then for all $\delta>0$,
\begin{displaymath}
\lima \mathcal{P}_V\left(\sup_{a_r\leq x\leq b_r}\left|\frac{L_{X_c}(e^{c h(c)},\mathfrak m_1+x)}{e^{c h(c)}}\frac{\int_{a_{r}}^{b_{r}}e^{-c V_{\mathfrak m_1}(y)}\ud y}{e^{-c V_{\mathfrak m_1}(x)}}-1\right|\leq \delta \right)=1
\end{displaymath}
and
\begin{displaymath}
\lima \mathcal{P}_V\left(\sup_{x\in\R\setminus[a_r,b_r]}L_{X_c}(e^{c h(c)},\mathfrak m_1+x)\leq \frac{\delta e^{c h(c)}}{\int_{a_{r}}^{b_{r}}e^{-c V_{\mathfrak m_1}(y)}\ud y} \right)=1.
\end{displaymath}
\end{prop}

\begin{proof}
The ideas of the proof are essentially the same as in Sections 2.1 and 2.2 of \cite{ad:llltbd}. For the first one, we proceed as in Proposition 2.1 of \cite{ad:llltbd} with $[a_r,b_r]$ instead of a compact $[-K,K]$. 
The proof of the second formula is the same as the one of Proposition 2.5 of \cite{ad:llltbd} but instead of computing with the integral of the local time (see Formula (34) of \cite{ad:llltbd}), we use the supremum on $\R\setminus[a_r,b_r]$ of the same local time. In both cases, same arguments as the ones of Proposition 3.1 of \cite{ad:llltbd} used with the equivalent of Proposition 2.5 of \cite{ad:llltbd} allow to write the formula with a deterministic time instead of a random time.
\end{proof}

Denote by $m^*_c(t)$ the favorite point of $X_c$,
$$m^*_c(t)=\inf\left\{x\in\R,\ L_{X_c}(t,x)=\sup_{y\in\R}L_{X_c}(t,y)\right\}.$$
The previous proposition implies that this point is localized in the neighborhood of $\mathfrak{m}_1$:
\begin{prop}\label{cvgmc}
If $V$ is a good environment $\P$-a.s. then for any $\delta>0$ such that
\begin{equation}\label{condlim}
\lim_{\varepsilon\rightarrow0}\limsup_{c\rightarrow\infty}\mathbb{P}\left(\inf_{[a_{r},-\delta/c^\alpha]\cup[\delta/c^\alpha,b_{r}]}cV_{\mathfrak m_1}\leq\varepsilon\right)=0
\end{equation}
we have
 \begin{displaymath}
 \lima \mathcal{P}\left(\left|m^*_c(e^{ch(c)})-\mathfrak{m}_1\right|\leq \delta/c^\alpha \right)=1.
 \end{displaymath}
\end{prop}
\begin{proof}
 Let $\varepsilon\in(0,1/2)$, for any $c>1$, we define two events:
 $$
 A_c:=\left\{\sup_{a_r\leq x\leq b_r}\left|\frac{L_{X_c}(e^{c h(c)},\mathfrak m_1+x)}{e^{c h(c)}}\frac{\int_{a_{r}}^{b_{r}}e^{-c V_{\mathfrak m_1}(y)}\ud y}{e^{-c V_{\mathfrak m_1}(x)}}-1\right|\leq \varepsilon \right\}
 $$
 and
$$
 B_c:=\left\{\sup_{x\in\R\setminus[a_r,b_r]}L_{X_c}(e^{c h(c)},\mathfrak m_1+x)\leq \frac{\varepsilon e^{c h(c)}}{\int_{a_{r}}^{b_{r}}e^{-c V_{\mathfrak m_1}(y)}\ud y} \right\}.
 $$
 On $A_c\cap B_c$, we have the following inequality
 $$L_{X_c}(e^{ch(c)},m^*_c(e^{ch(c)}))\geq L_{X_c}(e^{ch(c)},\mathfrak{m}_1)\geq \frac{(1-\varepsilon)e^{c h(c)}}{\int_{a_{r}}^{b_{r}}e^{-c V_{\mathfrak m_1}(y)}\ud y}.$$
Thereby $m^*_c(e^{ch(c)})\in(a_r,b_r)$ and so
 $$L_{X_c}(e^{ch(c)},m^*_c(e^{ch(c)}))\leq \frac{(1+\varepsilon)e^{c h(c)}}{\int_{a_{r}}^{b_{r}}e^{-c V_{\mathfrak m_1}(y)}\ud y}e^{-c V_{\mathfrak m_1}(m^*_c(e^{ch(c)}))}.$$
 Necessarily, 
 $$c V_{\mathfrak m_1}(m^*_c(e^{ch(c)})\leq\log\frac{1+\varepsilon}{1-\varepsilon}.$$
 Denote 
 $$
 \mathcal{E}^\varepsilon_c=\left\{V\in\tV\ ,\ \inf_{[a_r,-\delta/c^\alpha]\cup[\delta/c^\alpha,b_r]}cV_{\mathfrak m_1}>\log\frac{1+\varepsilon}{1-\varepsilon}\right\}.
 $$
 If $V\in\mathcal{E}^\varepsilon_c$, then on $A_c\cap B_c$, 
 $$\left|m^*_c(e^{ch(c)})-\mathfrak{m}_1\right|\leq \delta/c^\alpha.$$
 Thus,
 \begin{align*}
  \mathcal{P}\left(\left|m^*_c(e^{ch(c)})-\mathfrak{m}_1\right|>\delta/c^\alpha \right)&\leq \mathcal{P}(\overline{A_c})+\mathcal{P}(\overline{B_c})+\mathcal{P}(\overline{\mathcal{E}^\varepsilon_c})
 \end{align*}
 and therefore
 \begin{align*}
  \limsup_{c\rightarrow\infty}\mathcal{P}\left(\left|m^*_c(e^{ch(c)})-\mathfrak{m}_1\right|>\delta/c^\alpha \right)&\leq\lim_{\varepsilon\rightarrow0}\limsup_{c\rightarrow\infty}\mathcal{P}(\overline{\mathcal{E}^\varepsilon_c})=0
 \end{align*}
according to Hypothesis \eqref{condlim}.
\end{proof}

Now we can come back to the process $L_{X}$ and prove the following theorem, main result of this section:

\begin{theo}\label{limL}
Fix a real number $r\in(0,1)$. If the environment $V$ is a good environment $\P$-a.s. and there exists $\alpha>0$ such that $V$ is $\alpha$-stable then for all $\delta>0$,
\begin{displaymath}
\lima \mathcal{P}\left(\sup_{a_{cr}\leq x\leq b_{cr}}\left|\frac{L_{X}(e^{c},\mathfrak m_c+x)}{e^{c}}\frac{\int_{a_{cr}}^{b_{cr}}e^{- V_{\mathfrak m_c}(y)}\ud y}{e^{-V_{\mathfrak m_c}(x)}}-1\right|\leq \delta \right)=1.
\end{displaymath}
and
\begin{displaymath}
\lima \mathcal{P}\left(\sup_{x\in\R\setminus[a_{cr},b_{cr}]}L_{X}(e^{c},\mathfrak m_c+x)\leq \frac{\delta e^{c}}{\int_{a_{cr}}^{b_{cr}}e^{-V_{\mathfrak m_c}(y)}\ud y} \right)=1.
\end{displaymath}
Recall the definition of $m^*$ of Theorem \ref{cvptfav}:
$$m^*(t)=\inf\left\{x\in\R,\ L_{X}(t,x)=\sup_{y\in\R}L_{X}(t,y)\right\}.$$
If for any $\delta>0$, we have
\begin{equation}\label{condlima}
\lim_{\varepsilon\rightarrow0}\limsup_{c\rightarrow\infty}\mathbb{P}\left(\inf_{[a_{cr},-\delta]\cup[\delta,b_{cr}]}V_{\mathfrak m_c}\leq\varepsilon\right)=0
\end{equation}
then 
 \begin{displaymath}
 \lima \mathcal{P}\left(\left|m^*(e^{c})-\mathfrak{m}_c\right|\leq \delta \right)=1.
 \end{displaymath}
\end{theo}

\begin{rem}
We deal in this article with stable Lévy processes and we will see later that they verify  Hypothesis \eqref{condlima}. Moreover, the previous theorem also applies when the environment $V$ is a fractional Brownian motion.
\end{rem}

\begin{proof}
We fix $c>0$ and we recall that $V^c(\cdot):=c^{-1}V(c^\alpha\cdot)$. By the stability of $V$, the processes $V$ and $V^c$ have same law. We first remark that, as the paths of $V$ are in $\tV$, $\mathcal{P}$-a.s., the standard valley is well defined and therefore 
 
\begin{align*}
\mathfrak m_1(V^c)&=c^{-\alpha}\mathfrak m_c(V),\\
 a_{ r}(V^c_{\mathfrak m_1(V^c)})&=c^{-\alpha}a_{c r}(V_{\mathfrak m_c(V)}),\\
 b_{r}(V^c_{\mathfrak m_1(V^c)})&=c ^{-\alpha} b_{c r}(V_{\mathfrak m_c(V)}),
\end{align*}
and for all $x\in\R$,
\begin{equation}\label{scaleV}
V^c_{\mathfrak m_1(V^c)}(x)=\frac{1}{c}V_{\mathfrak m_c(V)}(c^\alpha x).
\end{equation}

Thus, using the substitution $z=c^\alpha y$, we get that

\begin{equation*}
\int_{a_r}^{b_r}e^{-c V_{\mathfrak m_1}(y)}dy=c^{-\alpha}\int_{a_{cr}}^{b_{cr}}e^{V_{\mathfrak m_c}(z)}dz.
\end{equation*}

Now we replace $t$ by $c^{-2\alpha}e^{c}$ and $x$ by $\mathfrak m_1(V^c)+c^{-\alpha}x$ in Corollary \ref{eglx}, then we get that, conditionally on $V$,

$\displaystyle\left(L_{X_c( V^c,\cdot)}(c^{-2\alpha}e^{c},\mathfrak m_1(V^c)+\ax),x\in\R\right)$ and $\displaystyle\left(\frac{1}{c^\alpha}L_{X(V,\cdot)}(e^c,\mathfrak m_c(V)+x), x\in\R\right)$ have the same law.
Moreover the processes $V$ and $V^c$ have the same law, so we obtain that
 
$\displaystyle\left(L_{X_c( V,\cdot)}(c^{-2\alpha}e^{c},\mathfrak m_1(V)+\ax), x\in\R\right)$ and $\displaystyle\left(\frac{1}{c^\alpha}L_{X(V,\cdot)}(e^c,\mathfrak m_c(V)+x), x\in\R\right)$ have the same law without conditioning.
Thus for all $c>0$, $K>0$, $r\in(0,1)$ and $\delta>0$,
\begin{align*}
&\mathcal{P}\left(\sup_{a_{cr}\leq x\leq b_{cr}}\left|\frac{L_{X(V,\cdot)}(e^{c},\mathfrak m_c(V)+x)}{e^{c}}\frac{\int_{a_{c r}}^{b_{c r}}e^{-V_{\mathfrak m_c}(y)}\ud y}{e^{-V_{\mathfrak m_c}(x)}}-1\right|<\delta\right)\quad=\\
&\mathcal{P}\left(\sup_{a_r\leq x\leq b_r}\left|\frac{L_{X_c( V,\cdot)}(c^{-2\alpha}e^{c},\mathfrak m_1(V)+x)}{c^{-2\alpha}e^{c}}\frac{\int_{a_{r}}^{b_{r}}e^{-c V_{\mathfrak m_1}(y)}\ud y}{e^{-c V_{\mathfrak m_1}(x)}}-1\right|<\delta\right).
\end{align*}
To finish, it is enough to remark that $c^{-2\alpha}e^{c}=e^{c(1-\frac{2\alpha}{c}\log c)}$ and $\limab(1-\frac{2\alpha}{c}\log c)=1$. Then, using Proposition \ref{cle}, we obtain the first convergence of Theorem \ref{limL}. The other ones are proved in the same way.
\end{proof}

\subsection{Limit law of the renormalized local time}
\emph{Now we suppose again that our environment is an $\alpha$-stable Lévy process with $\alpha\in[1,2]$ and is not a pure drift}.

\begin{lem}
Hypothesis \reff{condlima} is true for such a Lévy process:
\begin{equation*}
\lim_{\varepsilon\rightarrow0}\limsup_{c\rightarrow\infty}\mathbb{P}\left(\inf_{[a_{cr},-\delta]\cup[\delta,b_{cr}]}V_{\mathfrak m_c}\leq\varepsilon\right)=0.
\end{equation*}
Then, by Theorem \ref{limL}, we have the convergence in probability of $m^*(e^{c})-\mathfrak{m}_c$ to $0$.
\end{lem}
\begin{proof}
Recall that $\mathfrak{m}_c=m_c^+$ or $\mathfrak{m}_c=m_c^-$. As the proof is similar in the both cases, we only study the first one. If $\beta<\alpha$ then we use the convention $[\alpha,\beta]=\emptyset$ and $\inf_{[\alpha,\beta]}V_{\mathfrak m_c}=+\infty$.

\begin{align*}
&\mathbb{P}\left(\inf_{[a_{cr},-\mathfrak{m}_c]\cup[-\mathfrak{m}_c,-\delta]\cup[\delta,b_{cr}]}V_{\mathfrak m_c}\leq\varepsilon,\mathfrak{m}_c=m_c^+\right) \leq \\
&\mathbb{P}\left(\inf_{[a_{cr},-\mathfrak{m}_c]}V_{\mathfrak m_c}\leq\varepsilon,\mathfrak{m}_c=m_c^+\right)
+ \mathbb{P}\left(\inf_{[-m^+_c,-\delta]}V_{m^+_c}\leq\varepsilon\right)  
+ \mathbb{P}\left(\inf_{[\delta,b_{cr}]}V_{m^+_c}\leq\varepsilon \right).
\end{align*}

We use the stability of $V$, $\mathfrak{m}_c$, $a_{cr}$ and $b_{cr}$ (see \reff{scaleV} above) to get for any $\varepsilon>0$:

\begin{align*}
\mathbb{P}\left(\inf_{[a_{cr},-\mathfrak{m}_c]}V_{\mathfrak m_c}\leq\varepsilon,\mathfrak{m}_c=m_c^+\right) & = \mathbb{P}\left(\inf_{[a_{r},-\mathfrak{m}_1]}V_{\mathfrak m_1}\leq \varepsilon c^{-1},\mathfrak{m}_1=m_1^+\right)\\
& \xrightarrow[c\rightarrow \infty]{} \mathbb{P}\left(\inf_{[a_{r},-\mathfrak{m}_1]}V_{\mathfrak m_1}\leq\ 0,\mathfrak{m}_1=m_1^+\right).
\end{align*}
By uniqueness of the minima (Lemma \ref{continuite}), this last probability is equal to 0. 

$\P$-a.s. $\inf_{[\delta,+\infty)}\hat{V}^\uparrow>0$, then Point (1) of Proposition \ref{vallee} gives us that 
$$\lim_{\varepsilon\rightarrow0}\limsup_{c\rightarrow\infty}\mathbb{P}\left(\inf_{[-m^+_c,-\delta]}V_{m^+_c}\leq\varepsilon\right)=0.$$

Using now the stability of $V$ and then Point (2) of Proposition \ref{vallee}, we obtain:

\begin{align*}
& \mathbb{P}\left(\inf_{[\delta,b_{cr}]}V_{m^+_c}\leq\varepsilon\right)  = \mathcal{P}\left(\inf_{[c^{-\alpha}\delta,b_{r}]}V_{m^+_1}\leq \varepsilon c^{-1}\right) =\int \ind_{\left\{ \inf_{[c^{-\alpha}\delta,\tau_{r}(\omega)]}\omega\leq \varepsilon c^{-1} \right\}} \underset{\rightarrow}{\P}^1(d\omega)\\
&=\int \ind_{\left\{ \inf_{[c^{-\alpha}\delta,\tau_{r}(\omega)]}\omega\leq \varepsilon c^{-1} \right\}} f_1\left(\omega(\tau_1(\omega))\right) \P^{\uparrow,1}(d\omega) \leq \frac{\Gamma(\alpha\rho+1)}{\mathcal N(\tau_1<\infty)} \mathbb{P}\left(\inf_{[\delta,\tau_{cr}(V^\uparrow)]}V^\uparrow\leq\varepsilon\right)
\end{align*}
where the last inequality comes from the expression of $f_1$ given by \reff{expressionf1}. Finally,

$$\lim_{\varepsilon\rightarrow0}\limsup_{c\rightarrow\infty}\mathbb{P}\left(\inf_{[\delta,b_{cr}]}V_{m^+_c}\leq\varepsilon\right)\leq \frac{\Gamma(\alpha\rho+1)}{\mathcal N(\tau_1<\infty)} \mathbb{P}\left(\inf_{[\delta,\infty)}V^\uparrow=0\right)=0.$$
This finishes the proof.
\end{proof}

According to Theorem \ref{limL}, to obtain the weak convergence of the local time process, we only have to study the convergence of $e^{-V_{\mathfrak m_c}}/\int_{a_{rc}}^{b_{rc}} e^{-V_{\mathfrak m_c}}$, what we do now.

For $\omega\in\cv$, recall the definition of $\omega^+$ and $\omega^-$ given in \reff{w+w-} and notations of the first part: 
\begin{align*}
\underline{\tau}^+_c &= \inf\{ t\geq 0 ; V^+(t)-\underline{V}^+(t) \geq c \},\\
m^+_c &= \sup\{ t\leq \underline{\tau}^+_c ; V^+(t) - \underline{V}^+(t) = 0 \}\textrm{ and}\\
J_c^+&=\left(V^+(m^+_c)+c\right)\vee\overline{V}^+(m^+_c)
\end{align*}
and similar variables defined using $V^-$,  $\underline{\tau}^-_c$, $m^-_c$ and $J^-_c$. As said before, we can prove that the minimum of the standard valley  $\mathfrak m_c$ is equal to $m^+_c$ if $J^+_c<J^-_c$ and $-m^-_c$ otherwise.

\begin{prop}\label{conv exp}
Let $V$ be a stable Lévy process with index $\alpha\in[1,2]$ which is not a pure drift. We fix $r\in(0,1)$.

When $c\rightarrow\infty$, the process $\displaystyle \frac{e^{-V_{\mathfrak m_c}}}{\int_{a_{rc}}^{b_{rc}} e^{-V_{\mathfrak m_c}(y)}dy  }$ converges weakly (in the Skorohod sense) to $\displaystyle \frac{e^{-\tilde{V}}}{\int_{-\infty}^{+\infty} e^{-\tilde{V}(y)}dy  }$ where the law of $\tilde{V}$ is $\tilde{\P}$ and, under $\tilde{\P}(d\omega)$, $\omega^+$ and $\omega^-$ are independent, and distributed respectively as $\P^\uparrow$, the law of the Lévy process $V$ conditioned to stay positive, and $\hat{\P}^\uparrow$ the law of the dual process $V^-$ conditioned to stay positive.
\end{prop}

In order to prove Proposition \ref{conv exp}, we need the following Lemmas. In the first one, we prove that the integral $\int_{-\infty}^{+\infty} e^{-\tilde{V}(y)}dy $ is finite.

\begin{lem}
 We have  
$$\hbox{$\tilde{\P}$-a.s. }\int_{-\infty}^{+\infty} e^{- \tilde{V}(y)} dy <\infty.$$
\end{lem}

\begin{proof}
It is enough to prove the integrability on $\R_+$ for the process $(\hat{V}^\uparrow(x),x\geq0)$ whose law is $\hat{\P}^{\uparrow}$. Fix $t\geq0$ and recall that $U(t) = -V(L^{-1}(t))$. We prove the integrability of $e^{-U_t}$. The process $U$ is a subordinator, then it is characterized by its Laplace exponent $\psi_U$ (see Section III.1, \cite{b:pl}): for all $\lambda\geq0$,
$$\E\left[  e^{-\lambda U(t)} \right] = e^{-t\psi_U(\lambda)}.$$

For all $t>0$, $U(t)>0$ then $\psi_U(1)> 0$, we obtain that $\E\left[  e^{- U(t)} \right] = e^{-t\psi_U(1)}$ is integrable on $\R_+$. 
We recall the notation of the Proposition \ref{vallee}, we get
$$U(t) = - \underline{V}(d(t)) \leq \mathcal R_{V-\underline{V}}(t) - \underline{V}(d(t)).$$

Thanks to this Proposition, the process $(\mathcal R_{V-\underline{V}}(t) - \underline{V}(d(t)) ; t\geq 0)$ is equal in law to the process $\hat{V}^{\uparrow}$. 
Thus, $\E\left[  e^{- \hat{V}^{\uparrow}(t)} \right]$ is also integrable on $\R_+$, we deduce that $e^{- \hat{V}^{\uparrow}}$ is $\hat{\P}^{\uparrow}$-a.s. integrable.\\
\end{proof}

\begin{lem}\label{equivalence integrale}
Let $\omega$ be a non-negative càd-làg function with $\omega(0)=0$ which is continuous in its extrema and such that $0$ is the unique point of minimum of $\omega$. Then for all $a\leq \alpha< 0< \beta \leq b$, we get
$$\int_a^b e^{-c \omega(y)}dy \underset{c\rightarrow +\infty}{\sim} \int_\alpha^\beta e^{-c \omega(y)}dy.$$
\end{lem}

\begin{proof}
\begin{align*}
\int_a^b e^{-c \omega(y)}dy - \int_\alpha^\beta e^{-c \omega(y)}dy &=\int_a^\alpha e^{-c \omega(y)}dy + \int_\beta^b e^{-c \omega(y)}dy\\
&\leq |\alpha-a| e^{-c \inf_{y\in(a,\alpha)}\omega(y)} + |b-\beta| e^{-c \inf_{y\in(\beta,b)}\omega(y)}.
\end{align*}
As $\omega$ is positive on $\R\setminus\{0\}$ and continuous at its extrema, for all $x\neq0$, $\inf_{(a,\alpha)}\omega>0$ and $\inf_{(\beta, b)}\omega>0$ (see Remark \ref{continuite ext locaux}), then $\lim_{c\rightarrow +\infty} \int_a^b e^{-c \omega(y)}dy - \int_\alpha^\beta e^{-c \omega(y)}dy =0$.

By the Laplace method,
$$\lim_{c\rightarrow +\infty} \frac{1}{c} \log \int_{a}^{b} e^{-c\omega(y)}dy = \sup_{(a,b)}(-\omega).$$
As $a\leq 0\leq b$ then $\sup_{(a,b)}(-\omega)=0$, thus $\int_a^b e^{-c\omega(y)}dy = e^{o(c)}$.
$$\lim_{c\rightarrow +\infty}   \frac{\int_a^b e^{-c\omega(y)}dy-\int_\alpha^\beta e^{-c\omega(y)}dy}{\int_a^b e^{-c\omega(y)}dy} \leq \lim_{c\rightarrow +\infty}e^{-c\inf_{(a,\alpha)}\omega-c\inf_{(\beta,b)}\omega-o(c)}=0.$$
This implies the equivalence.
\end{proof}

\begin{proof}[Proof of Proposition \ref{conv exp}]
We follow the same idea as in the proof of Proposition \ref{convergence}.

It is enough to prove that for a bounded continuous function $F$ in the Skorohod topology and such that $F(\omega)=F(\omega\ind_{[-K,K]})$,
$$\lim_{c\rightarrow +\infty} \E\left[ F\left( \frac{e^{-V_{\mathfrak m_c}}}{\int_{a_{rc}}^{b_{rc}} e^{-V_{\mathfrak m_c}(y)}dy  } \right) \right] = \E\left[ F\left( \frac{e^{-\tilde{V}}}{\int_{-\infty}^{+\infty} e^{-\tilde{V}(y)}dy  } \right)  \right].$$

Using Formula (\ref{scaleV}) and the fact that $\mathfrak m_1\in\{-m_1^-,m_1^+\}$, 
\begin{flushleft}
$\displaystyle \E\left[ F\left( \frac{e^{-V_{\mathfrak m_c}}}{\int_{a_{rc}}^{b_{rc}} e^{-V_{\mathfrak m_c}(y)}dy  } \right) \right] =  \E\left[ F\left( \frac{e^{- c V_{\mathfrak m_1}(c^{-\alpha} \cdot) }}{c^\alpha\int_{a_{r}}^{b_{r}} e^{-c V_{\mathfrak m_1}(y)}dy  } \right) , \mathfrak m_1 = m_1^+\right] +$
\end{flushleft}
\begin{flushright}
$\displaystyle\E\left[ F\left( \frac{e^{- c V_{\mathfrak m_1}(c^{-\alpha} \cdot) }}{c^\alpha\int_{a_{r}}^{b_{r}} e^{-c V_{\mathfrak m_1}(y)}dy  } \right) , \mathfrak m_1 = -m_1^-\right].$
\end{flushright}

We will compute the two terms of the sum by the same method, thus we only study the first one. Thanks to Lemma \ref{equivalence integrale}, it suffices to take the integral from $-\mathfrak m_1\vee a_r$ to $b_r$.
By definition, the law of the process $\left(\tilde{V}(t),t\geq0\right)$ is $\P^\uparrow$ and the law of the process $\left(\tilde{V}((-t)-),t\geq0\right)$ is $\hat{\P}^\uparrow$, moreover they are independent. Then we write them respectively $V^\uparrow$ and $\hat{V}^\uparrow$. Using the fact that $\mathfrak m_1=m_1^+$ if and only if $J^+_1<J^-_1$ and the result of Proposition \ref{vallee}, we get that

\begin{flushleft}
$\displaystyle \E\left[ F\left( \frac{e^{- c V_{\mathfrak m_1}(c^{-\alpha} \cdot) }}{c^\alpha\int_{-\mathfrak m_1\vee a_{r}}^{b_{r}} e^{-c V_{\mathfrak m_1}(y)}dy  } \right) , \mathfrak m_1=m_1^+ , \frac{K}{c^{\alpha}}\leq \mathfrak m_1\wedge \tau_1(\underset{\rightarrow}{V}) \right] = $
\end{flushleft}
\begin{flushright}
\begin{equation}\label{reprise}
\displaystyle \hspace{2cm}\E\left[ F\left( \frac{e^{- c \tilde{V}(c^{-\alpha} \cdot) }}{c^\alpha\int_{-\tilde{a}_{r}}^{\tilde{b}_{r}} e^{-c \tilde{V}(y)}dy  } \right) f_1\left(V^\uparrow\left(\tau_1  \right) \right) , \tilde{J}_1^+\leq\tilde{J}_1^-, \frac{K}{c^{\alpha}}\leq \hat{m}^\uparrow_1\wedge \tau_1(V^\uparrow) \right]
\end{equation}
\end{flushright}
where $\tilde{a}_{r}$, $\tilde{b}_{r}$, $\tilde{J}_1^+$ and $\tilde{J}_1^-$ are defined similarly as $a_r$, $b_r$, $J_1^+$ and $J^-_1$ but using the process $\tilde{V}$ :
\begin{align*}
    \tilde{a}_{r}&:=\hat{m}^\uparrow_1\vee \inf\left\{x\geq0/\hat{V}^\uparrow(x)>r\right\},\\
    \tilde{b}_{r}&:=\inf\left\{x\geq0/V^\uparrow(x)>r\right\},\\
\tilde{J}_1^+&:= \left(1-\hat{V}^\uparrow(\hat{m}^\uparrow_1)\right)\vee \left(\sup_{(0,\hat{m}^\uparrow_1)}{\hat{V}^\uparrow} - \hat{V}^\uparrow(\hat{m}^\uparrow_1)\right),\\
\textrm{and }\tilde{J}_1^-&:= \left(1-W^{\uparrow}(m^{\uparrow}_1)\right)\vee \left(\sup_{(0,m^{\uparrow}_1)}{W^{\uparrow}} - W^{\uparrow}(m^{\uparrow}_1)\right)
   \end{align*}
where $W^{\uparrow}$ is a process with law $\P^\uparrow$ independent of $\tilde{V}$. To simplify the notation, we write $F_{\tilde{a}_r,\tilde{b}_r}$ for $F\left( \frac{e^{- c \tilde{V}(c^{-\alpha} \cdot) }}{c^\alpha\int_{-\tilde{a}_{r}}^{\tilde{b}_{r}} e^{-c \tilde{V}(y)}dy  } \right)$.

As in the proof of Proposition \ref{convergence}, we fix $\varepsilon>0$ and we separate the several terms of the expectation with respect to $(\tilde{V}(t), t \leq -\sigma^1_\varepsilon)$, $(\tilde{V}(t),-\sigma^1_\varepsilon \leq t\leq \sigma^2_\varepsilon)$ and $(\tilde{V}(t),\sigma^2_\varepsilon\leq t)$, where:
$$
\sigma^1_\varepsilon=\sigma_\varepsilon(\hat{V}^\uparrow)=\inf\left\{t>0/{\hat{V}^\uparrow}(t)-\uu{\hat{V}^\uparrow}(t)=0 \textrm{ and } \exists s<t, \hat{V}^\uparrow(s)-\uu{\hat{V}^\uparrow}(s)>\varepsilon\right\},
$$
$$
\sigma^2_\varepsilon=\sigma_\varepsilon(V^\uparrow)=\inf\left\{t>0/{V^\uparrow}(t)-\uu{V^\uparrow}(t)=0 \textrm{ and } \exists s<t, V^\uparrow(s)-\uu{V^\uparrow}(s)>\varepsilon\right\},
$$

to obtain the asymptotic independence of $F_{\tilde{a}_r,\tilde{b}_r}$, of $V^\uparrow(\tau_1)$ and of $\tilde{J}^+_1$.

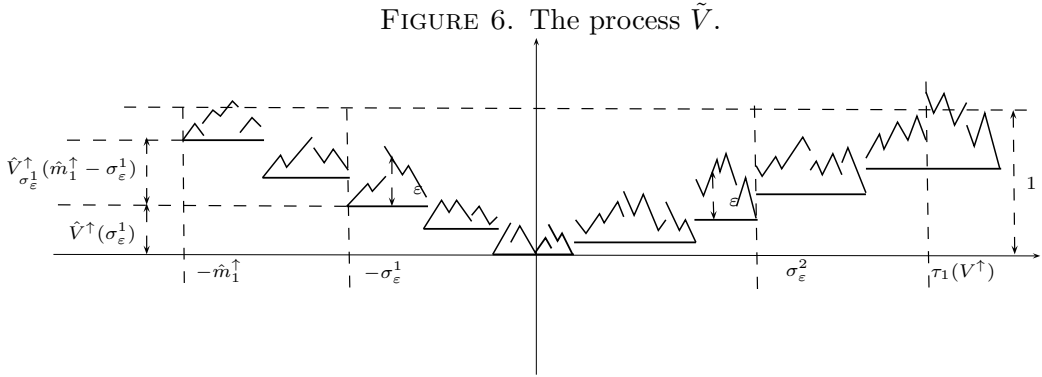
\begin{center}
\begin{figure}[ht]
\caption{The process $\tilde{V}$.}

\hspace{-4cm}
\scalebox{1} % Change this value to rescale the drawing.
{
\begin{pspicture}(0,-2.245)(17.0875,2.245)
\psline[linewidth=0.01cm,arrowsize=0.05291667cm 2.0,arrowlength=1.4,arrowinset=0.4]{<-}(10.288438,2.24)(10.288438,-2.24)
\psline[linewidth=0.01cm,arrowsize=0.05291667cm 2.0,arrowlength=1.4,arrowinset=0.4]{->}(3.8684375,-0.64)(16.988438,-0.66)
\psline[linewidth=0.024cm](13.230437,-0.18)(12.390437,-0.18)
\psline[linewidth=0.02](13.210438,-0.16)(13.070437,0.38)(12.950438,0.0)
\psline[linewidth=0.02](12.950438,0.28)(12.810437,0.7)(12.730437,0.46)(12.690437,0.62)(12.510438,0.14)(12.410438,0.36)
\psline[linewidth=0.024cm](12.410438,-0.48)(10.790438,-0.48)
\psline[linewidth=0.02](12.410438,-0.48)(12.270437,-0.18)(12.190437,-0.32)(12.010438,0.06)
\psline[linewidth=0.02](12.010438,-0.18)(11.910438,-0.32)(11.710438,0.16)(11.588437,-0.12)(11.448438,0.2)
\psline[linewidth=0.02](11.450438,-0.14)(11.370438,0.06)(11.190437,-0.22)(11.128437,-0.1)(10.948438,-0.36)(10.810437,-0.18)
\psline[linewidth=0.024cm](16.468437,0.5)(14.670438,0.5)
\psline[linewidth=0.02](16.450438,0.5)(16.250437,1.24)(16.110437,0.8)(16.010437,1.08)
\psline[linewidth=0.02](16.008438,1.38)(15.848437,1.08)(15.708438,1.5)(15.570437,1.24)(15.468437,1.52)
\psline[linewidth=0.02](15.468437,0.88)(15.348437,1.2)(15.228437,0.88)(15.090438,1.12)(14.910438,0.76)(14.850437,0.96)(14.670438,0.64)
\psline[linewidth=0.024cm](14.670438,0.16)(13.210438,0.16)
\psline[linewidth=0.02](14.670438,0.18)(14.430437,0.78)(14.350437,0.58)
\psline[linewidth=0.02](14.350437,0.42)(14.230437,0.7)(14.090438,0.42)(14.010438,0.6)(13.930437,0.48)
\psline[linewidth=0.02](13.930437,0.74)(13.830438,0.92)(13.550438,0.5)(13.470437,0.66)(13.290438,0.4)(13.210438,0.54)
\psline[linewidth=0.024](10.770437,-0.64)(10.630438,-0.36)(10.570437,-0.46)(10.450438,-0.24)
\psline[linewidth=0.024](10.4904375,-0.56)(10.410438,-0.42)(10.270437,-0.62)
\psline[linewidth=0.024cm](10.788438,-0.64)(10.290438,-0.64)
\psline[linewidth=0.02](10.288438,-0.64)(10.068438,-0.26)(9.948438,-0.48)
\psline[linewidth=0.02](9.948438,-0.22)(9.708438,-0.64)(9.708438,-0.64)
\psline[linewidth=0.02](9.768437,-0.14)(9.668438,-0.24)(9.528438,0.0)(9.428437,-0.12)
\psline[linewidth=0.02](9.428437,-0.2)(9.228437,0.08)(9.128437,-0.08)(9.028438,0.08)(8.788438,-0.3)
\psline[linewidth=0.02](6.6284375,1.06)(6.5084376,1.18)(6.3484373,0.98)
\psline[linewidth=0.02](6.3284373,1.32)(6.2484374,1.4)(6.0884376,1.2)(5.9884377,1.28)(5.8484373,1.1)
\psline[linewidth=0.02](5.8684373,0.94)(5.7484374,1.1)(5.5884376,0.88)
\psline[linewidth=0.024cm](5.5884376,0.88)(6.6684375,0.88)
\psline[linewidth=0.024cm](8.788438,-0.3)(9.788438,-0.3)
\psline[linewidth=0.024cm](9.708438,-0.64)(10.288438,-0.64)
\psline[linewidth=0.02](8.828438,0.12)(8.588437,0.54)(8.508437,0.42)(8.268437,0.8)
\psline[linewidth=0.02](8.288438,0.4)(8.108438,0.2)(8.048437,0.3)(7.7684374,0.0)
\psline[linewidth=0.024cm](7.7684374,0.0)(8.848437,0.0)
\psline[linewidth=0.016cm,linestyle=dashed,dash=0.16cm 0.16cm,arrowsize=0.05291667cm 2.0,arrowlength=1.4,arrowinset=0.4]{<->}(8.368438,0.68)(8.368438,0.0)
\usefont{T1}{ptm}{m}{n}
\rput(8.722187,0.225){\tiny $\varepsilon$}
\psline[linewidth=0.02](7.7884374,0.48)(7.5884376,0.76)(7.4684377,0.58)(7.3284373,0.78)
\psline[linewidth=0.02](7.3084373,0.92)(6.9684377,0.52)(6.8884373,0.68)(6.6484375,0.38)
\psline[linewidth=0.024cm](6.6484375,0.38)(7.8084373,0.38)
\psline[linewidth=0.016cm,linestyle=dashed,dash=0.16cm 0.16cm](5.5884376,1.32)(5.6084375,-1.02)
\psline[linewidth=0.016cm,linestyle=dashed,dash=0.16cm 0.16cm,arrowsize=0.05291667cm 2.0,arrowlength=1.4,arrowinset=0.4]{<->}(12.628437,0.48)(12.6484375,-0.18)
\psline[linewidth=0.016cm,linestyle=dashed,dash=0.16cm 0.16cm](7.7684374,1.3)(7.8084373,-1.08)
\psline[linewidth=0.016cm,linestyle=dashed,dash=0.16cm 0.16cm](13.188437,1.34)(13.228437,-1.1)
\psline[linewidth=0.016cm,linestyle=dashed,dash=0.16cm 0.16cm](15.468437,1.66)(15.508437,-1.08)
\psline[linewidth=0.016cm,linestyle=dashed,dash=0.16cm 0.16cm](16.808437,1.28)(4.8084373,1.32)
\usefont{T1}{ptm}{m}{n}
\rput(12.922188,0.065){\tiny $\varepsilon$}
\usefont{T1}{ptm}{m}{n}
\rput(8.262188,-0.875){\tiny $-\sigma^1_\varepsilon$}
\usefont{T1}{ptm}{m}{n}
\rput(13.762188,-0.855){\tiny $\sigma^2_\varepsilon$}
\usefont{T1}{ptm}{m}{n}
\rput(6.0721874,-0.855){\tiny $-\hat{m}^\uparrow_1$}
\usefont{T1}{ptm}{m}{n}
\rput(15.962188,-0.855){\tiny $\tau_1(V^\uparrow)$}
\psline[linewidth=0.016cm,linestyle=dashed,dash=0.16cm 0.16cm](7.7684374,0.0)(3.9284375,0.02)
\psline[linewidth=0.016cm,linestyle=dashed,dash=0.16cm 0.16cm](5.6084375,0.88)(4.0484376,0.88)
\psline[linewidth=0.016cm,linestyle=dashed,dash=0.16cm 0.16cm,arrowsize=0.05291667cm 2.0,arrowlength=1.4,arrowinset=0.4]{<->}(5.1084375,0.02)(5.1084375,0.92)
\psline[linewidth=0.016cm,linestyle=dashed,dash=0.16cm 0.16cm,arrowsize=0.05291667cm 2.0,arrowlength=1.4,arrowinset=0.4]{<->}(5.1084375,0.04)(5.1084375,-0.64)
\usefont{T1}{ptm}{m}{n}
\rput(4.5221877,-0.355){\tiny $\hat{V}^\uparrow(\sigma^1_\varepsilon)$}
\usefont{T1}{ptm}{m}{n}
\rput(4.1121874,0.445){\tiny $\hat{V}^\uparrow_{\sigma^1_\varepsilon}(\hat{m}_1^\uparrow - \sigma^1_\varepsilon)$}
\psline[linewidth=0.016cm,linestyle=dashed,dash=0.16cm 0.16cm,arrowsize=0.05291667cm 2.0,arrowlength=1.4,arrowinset=0.4]{<->}(16.628437,1.28)(16.648438,-0.64)
\usefont{T1}{ptm}{m}{n}
\rput(16.872187,0.305){\tiny $1$}
\end{pspicture} 
}

\end{figure}
\end{center}

Study in detail the hypothesis $\tilde{J}^+_1 \leq \tilde{J}^-_1 $. Denote 
$$M_+ = \sup_{(0,\hat{m}^\uparrow_1)}\hat{V}^\uparrow  -  \hat{V}^\uparrow(\hat{m}_1^\uparrow)\hbox{ and }M_\varepsilon = \sup_{(0,\sigma^1_\varepsilon)}\hat{V}^\uparrow-\hat{V}^\uparrow(\hat{m}_1^\uparrow).$$

For $\varepsilon$ small enough, $M_\varepsilon < M_+$, in this case: $\sigma^1_\varepsilon < \hat{m}^\uparrow_1$, $\sup_{(0,\hat{m}^\uparrow_1)}\hat{V}^\uparrow = \sup_{(\sigma^1_\varepsilon,\hat{m}^\uparrow_1)}\hat{V}^\uparrow$ and $\hat{V}^\uparrow(\hat{m}^\uparrow_1) = \hat{V}^\uparrow(\sigma^1_\varepsilon) + \hat{V}^\uparrow_{\sigma^1_\varepsilon}(\hat{m}_1^\uparrow - \sigma^1_\varepsilon)$. Moreover, for $\varepsilon$ small enough, $\sigma^2_\varepsilon < \tau_1(V^\uparrow)$, thus $$\tilde{J}^+_1 = \phi_\varepsilon\left(\hat{V}^\uparrow(\sigma^1_\varepsilon) \right) \hbox{ with }\phi_\varepsilon(x) = \left(1-   x - \hat{V}^\uparrow_{\sigma^1_\varepsilon}(\hat{m}_1^\uparrow - \sigma^1_\varepsilon) \right)\vee \left(  \sup_{(\sigma^1_\varepsilon,\hat{m}^\uparrow_1)}\hat{V}^\uparrow_{\sigma^1_\varepsilon} + \hat{V}^\uparrow_{\sigma^1_\varepsilon}(\hat{m}^\uparrow_1-\sigma^1_\varepsilon) \right).$$

We follow the idea of the proof of Proposition \ref{convergence} with Lemma \ref{levy renormalise}. We use Lemma \ref{equivalence integrale} for the integral in the function $F$ and we remark that 
$$\lim_{c\rightarrow +\infty} \P\left( \frac{K}{c^{\alpha}} \leq \hat{m}^\uparrow_1\wedge \tau_1(V^\uparrow) \right)=\lim_{c\rightarrow +\infty} \P\left( \frac{K}{c^{\alpha}}\leq \sigma^1_\varepsilon \wedge \sigma^2_\varepsilon \right)=1.$$ 
Recall \reff{reprise}, we obtain

\begin{flushleft}
$
\displaystyle \lim_{c\rightarrow +\infty}\E\left[ F_{\tilde{a}_r,\tilde{b}_r}\; f_1\left(V^\uparrow\left(\tau_1  \right) \right) , \tilde{J}_1^+\leq\tilde{J}_1^-, \frac{K}{c^{\alpha}}\leq \hat{m}^\uparrow_1\wedge \tau_1(V^\uparrow) \right] =
$
\end{flushleft}

\begin{flushright}
$
\displaystyle \lim_{\varepsilon\rightarrow 0}\lim_{c\rightarrow +\infty} \E \left[ F_{\sigma_\varepsilon^1,\sigma_\varepsilon^2} \;   \chi_2\left(V^\uparrow(\sigma_\varepsilon^2)\right)\,\chi_1\left(\hat{V}^\uparrow(\sigma_\varepsilon^1)\right) , \frac{K}{c^\alpha}\leq  \sigma_\varepsilon^1\wedge \sigma_\varepsilon^2, \sigma_\varepsilon^1\leq \hat{m}^\uparrow_1, \sigma_\varepsilon^2\leq \tau_1(V^\uparrow) \right]
$
\end{flushright}
where
\begin{align*}
\chi_1(x) &= \P\left( \left(1-   x - \hat{V}^\uparrow(\hat{m}_1^\uparrow) \right)\vee M_+    <\tilde{J}^-_1 \right) \hbox{ and}\\
\chi_2(x) &= \E\left[  f_1\left(x+V^\uparrow(\tau_{1-x})\right) \right].
\end{align*}
Moreover, $\lim_{\varepsilon\rightarrow 0} \chi_1\left(\hat{V}^\uparrow(\sigma_\varepsilon^1)\right) = \P(\tilde{J}^+_1 \leq \tilde{J}^-_1)=\P(\mathfrak m_1>0)$ and $\lim_{\varepsilon\rightarrow 0} \chi_2\left(V^\uparrow(\sigma_\varepsilon^2)\right) = 1$ thanks to \reff{limite psi}. Then we obtain:

$$\lim_{\varepsilon\rightarrow 0}\lim_{c\rightarrow +\infty} \left| \E\left[ F_{\tilde{a}_r,\tilde{b}_r} \left(\P(\mathfrak m_1>0)  -    \chi_1\left(\hat{V}^\uparrow(\sigma_\varepsilon^1)\right) \,  \chi_2\left(V^\uparrow(\sigma_\varepsilon^2)\right) \right)\right] \right|=0.$$

We use the stability of $V^\uparrow$ and of $\hat{V}^\uparrow$ (see Lemma \ref{stabiliteuparrow}). Moreover, for all $r\in(0,1)$, $\displaystyle \lim_{c\rightarrow +\infty} \tilde{a}_{cr}=-\infty$ and $\displaystyle \lim_{c\rightarrow +\infty} \tilde{b}_{cr}=+\infty$  a.s., then $\frac{e^{-\tilde{V}}}{\int_{a_{rc}}^{b_{rc}} e^{-\tilde{V}(y)}dy  }$ converges to $\frac{e^{-\tilde{V}}}{\int_{-\infty}^{+\infty} e^{-\tilde{V}(y)}dy  }$ in the Skorohod topology on $[-K,K]$. Thus 
$$ \lim_{c\rightarrow +\infty} \E\left[ F\left( \frac{e^{- c V_{\mathfrak m_1}(c^{-\alpha} \cdot) }}{c^\alpha\int_{a_{r}}^{b_{r}} e^{-c V_{\mathfrak m_1}(y)}dy  } \right) ,\mathfrak m_1 = m_1^+\right]  =\E\left[ F\left( \frac{e^{-\tilde{V}}}{\int_{-\infty}^{+\infty} e^{-\tilde{V}(y)}dy  } \right)  \right]\P(\mathfrak m_1>0).$$

We similarly study the case where $\mathfrak m_1=-m_1^-$, we obtain:

$$\lim_{c\rightarrow +\infty} \E\left[ F\left( \frac{e^{- c V_{\mathfrak m_1}(c^{-\alpha} \cdot) }}{c^\alpha\int_{a_{r}}^{b_{r}} e^{-c V_{\mathfrak m_1}(y)}dy  } \right) , \mathfrak m_1 = -m_1^-\right] =\E\left[ F\left( \frac{e^{-\tilde{V}}}{\int_{-\infty}^{+\infty} e^{-\tilde{V}(y)}dy  } \right)  \right]\P(\mathfrak m_1<0).$$

Then we obtain the result.

\end{proof}

Using Theorem \ref{limL} and Proposition \ref{conv exp}, we deduce Theorem \ref{cvloi} given in the introduction.

\vspace{1cm}
\textbf{Acknowledgments: }
 We are grateful to Romain Abraham and Pierre Andreoletti for helpful discussions. We also want to thank an anonymous referee for careful reading and many suggestions.

\end{document}